\documentclass[letterpaper,12pt]{article}

\usepackage{diagrams}
\usepackage[top=1.0in,right=1.0in,left=1.0in,bottom=1.0in]{geometry}
\usepackage{enumerate}
\usepackage{psfrag}
\usepackage{pstricks,pst-node}
\usepackage{hyperref}

\usepackage{psfrag}

\usepackage{amsmath}
\usepackage{amssymb}
\usepackage{amsthm}
\usepackage{slashed}

\theoremstyle{plain}
\newtheorem{theorem}{Theorem}[section]
\newtheorem{lemma}[theorem]{Lemma}

\theoremstyle{definition}
\newtheorem{defn}[theorem]{Definition}
\newtheorem*{dcp}{The Philosophy of \dag\-Categories}

\psset{nodesep=4pt,linewidth=\fontdimen8\textfont3,%
arrowinset=0, arrowsize=2.4pt 0, arrowlength=1.6, labelsep=3pt}

\newcommand\ignore[1]{}

\newcommand\mathmove[3]{\raisebox{#1}{\ensuremath{\hspace{#2}#3}}}

\newcommand\shiftbrackets[2]{\raisebox{-#1}{\ensuremath{\left( \raisebox{#1}{#2} \right)}}}

\newcommand\jskip{3pt}

\newcommand\xbigoplus{{\bigoplus}}
\newcommand\Ob{\ensuremath{\mathrm{Ob}}}

\newcommand\Tr{\mathrm{Tr}}

\newcommand{\jbeginproof}{\begin{proof}\vspace{-\jskip}}
\newcommand{\jbeginenumerate}{\begin{enumerate}\vspace{-\jskip}}
\newcommand{\jnoindent}{\vspace{-\jskip}\noindent}

\newcommand{\Hom}{\mathrm{Hom}}

\newcommand{\cat}[1]{\ensuremath{\mathbf{#1}}}
\newcommand{\id}{\ensuremath{ \mathrm{id} }}

\newcommand{\pb}{\;\!}
\newcommand{\ma}{\pb\!}

\newcommand\jbigl{\raisebox{-0.30pt}{\textrm{\large (}\hspace{-0.3pt}}}
\newcommand\jbigr{\hspace{-1.0pt}\mbox{\raisebox{-0.30pt}{\textrm{\large )}}}}

\renewcommand{\dag}{\ensuremath{\dagger}}
\newcommand\sdag\dag
\newcommand{\pdag}{{\ensuremath{\phantom{\sdag}}}}

\renewcommand{\to}{{\mbox{\begin{diagram}[width=4pt] {} & \rTo\end{diagram}}}\hspace{-1pt}}

\newcommand{\into}{\mbox{\begin{diagram}[width=4pt] {}&\rInto\end{diagram}}\ma}

\newarrow{mapsto}{|}{-}{-}{-}{>}
\renewcommand{\mapsto}{\mbox{\begin{diagram}[width=4pt] {}&\rMapsto\end{diagram}}\ma}

\begin{document}

\title{
Completeness of \dag-categories and\\the complex numbers
}

\author{Jamie Vicary
\\
Oxford University Computing Laboratory
\\
\texttt{jamie.vicary@comlab.ox.ac.uk}}

\date{July 8, 2010}
\maketitle
%-Title

\begin{abstract}
The complex numbers are an important part of quantum theory, but are difficult to motivate  from a theoretical perspective. We describe a simple formal framework for theories of physics, and show that if a theory of physics presented in this manner satisfies certain completeness properties, then it necessarily includes the complex numbers as a mathematical ingredient. Central to our approach are the techniques of category theory, and we introduce a new category-theoretical tool, called the \emph{\dag\-limit}, which governs the way in which systems can be combined to form larger systems. These \dag\-limits can be used to characterize the properties of the \dag\-functor on the category of finite-dimensional Hilbert spaces,  and so can be used as an equivalent definition of the inner product. One of our main results is that in a nontrivial monoidal \dag\-category  with finite \dag\-limits and a simple tensor unit, the semiring of scalars embeds into an involutive field of characteristic~0 and orderable fixed field.
\end{abstract}

\section{Introduction}
The purpose of this paper is to describe a set of properties of a theory of physics, which together imply that the theory makes use of the complex numbers. These properties are phrased in terms of the way that physical processes interact with each other, and as a result are intuitive and physical. The approach as a whole is a robust one; we are not concerned with many fine details of the theory, such as the nature of dynamics, or the way that measurement is described.

There is a vast literature of investigations into the mathematical foundations of quantum theory, which varies greatly in approach and perspective.   Some of this work tackles the problem of deriving the structure of quantum theory from physical or operational principles, a small sample of which is  \cite{a09-oqm, cbh03-cqt, da08-pt, di08-nsp, h01-qt5}. It is this type of research that has the strongest connection to the ideas presented here. There is also a large body of work investigating the properties of generalized quantum theories based on fields different to the complex numbers~\cite{bj96-esq, h96-hqm, la96-oqm, l06-wqp, v05-gqc}, against which the results presented here serve as a foil. 

To apply our method to a particular theory of physics, we first need to obtain from the theory a family of \emph{systems}, equipped with a family of \emph{processes} which go from one system to another. We will often denote processes as $f:A \to B$, which indicates a process $f$ going from system $A$ to system $B$. It is sometimes useful to imagine that systems are \emph{sets of states}, and that processes are \textit{functions} taking states of one system into states of another, but we will not rely on any such interpretation. For any two `head-to-tail' processes $f:A \to B$ and $g:B \to C$ we require that there exists a composite process $f;g:A \to C$, interpreted as the process $f$ followed by the process $g$. We require  that this composition is associative, and for any system $A$ we require the existence of a `trivial' process $\id_A:A \to A$ which is the identity for composition. These are exactly the axioms of a \emph{category}, and we will make essential use of the tools of category theory to prove our results.

We call this category the \emph{category of processes} associated to a particular theory. It will not necessarily completely define the theory; other important aspects, such as observation or measurement, are likely to be outside of its remit. Also,  very few real-world theories of physics will be naturally presented in terms of a category of processes, but for many theories there will nevertheless be natural candidates for such a category. If any of these candidates have the properties we will describe, that will indicate that the underlying theory somehow makes use of on the complex numbers. For the case of the theory of quantum mechanics, which certainly makes use of the complex numbers, we might define the category of processes to have separable Hilbert spaces as systems and bounded linear maps as processes, and the category obtained in this way satisfies the properties we will describe.

The first property that we require is that each process has an \emph{adjoint}, which can be considered as a formal `reversal' of the process. We use the term `adjoint' since this is a generalization of a familiar operation from quantum theory, taking the adjoint of a bounded linear map between Hilbert spaces. For any $f:A \to B$ its adjoint is a process $f ^\sdag:B \to A$; we require that $(f^\sdag) ^\sdag = f$ for any process $f$, and also that $(f;g) ^\sdag = g ^\sdag; f^\sdag$ for any composable processes $f$ and $g$. These properties define a \emph{functor} from our category to itself, and we call this the \emph{\dag\-functor}.  A second property that we require is \emph{superposition}: for any two parallel processes $f,g:A \to B$ there must exist a third process \mbox{$f+g:A \to B$}, where ${+}$ is an associative, unital, commutative operation with the property that \mbox{$(f+g);h=f;h+g;h$} for any $h:B \to C$ and any system $C$.

Finally, we require a notion of \emph{compound system}: for any two systems $A$ and $B$ there must exist a compound system $A \otimes B$, where ${\otimes}$ is an associative, unital operation.\footnote{Experts in category theory will note that we are describing a strict monoidal category here; a weak one would do just as well.} A useful intuition for this is the systems $A$ and $B$ existing simultaneously, but independently and without necessarily interacting. The unit for the compounding operation $\otimes$ is a system $I$, called the \emph{neutral} or \emph{unit} system, such that $I \otimes A = A = A \otimes I$ for all systems $A$. This compounding operation must also be defined on processes, so for all $f:A \to B$ and $g:C \to D$ there exists a process $f \otimes g: A \otimes C \to B \otimes D$; this must interact well with composition, satisfying the compatibility equation $(f \otimes g); (h \otimes j) = (f;h) \otimes (g;j)$ for all appropriate processes $f$, $g$, $h$ and $j$. If we interpret the process $f \otimes g$ as representing processes $f$ and $g$ occurring simultaneously and independently, then this compatibility equation makes intuitive sense: it says that performing $f$ and $g$ simultaneously, and then performing $h$ and $j$ simultaneously, is the same as performing $f$ followed by $h$, while simultaneously performing $g$ followed by~$j$.

Suppose now that we have a whole collection of systems and processes, of the following general form:
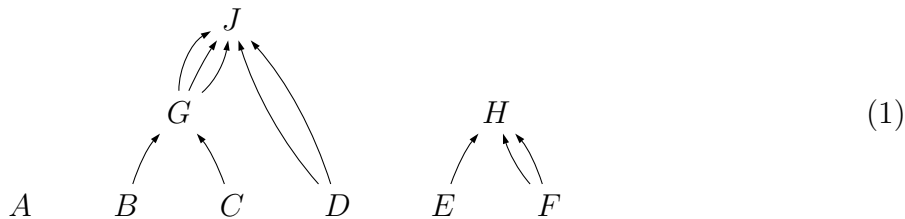
\begin{equation}
\label{diagexample}
\begin{diagram}[midshaft,width=20pt]
&&&&\rnode{G}{J}&&&&&&
\\
&&&\rnode{F}{G}&&&&&&\rnode{H}{H}&
\\
\rnode{J}{A} &&\rnode{A}{B} & & \rnode{B}{C} && \rnode{C}{D} & \hspace{0pt}&\rnode{D}{E} & & \rnode{E}{F}
%\\
%\\
%&&& {\gray\rnode{L}{L}} &&&&
\ncarc[arcangle=12]{->}{A}{F}
\ncarc[arcangle=30]{->}{F}{G}
\ncarc[arcangle=5]{->}{F}{G}
\ncarc[arcangle=-20]{->}{F}{G}
\ncarc[arcangle=-10]{->}{B}{F}
\ncarc[arcangle=12]{->}{C}{G}
\ncarc[arcangle=-12]{->}{C}{G}
\ncarc[arcangle=10]{->}{D}{H}
\ncarc[arcangle=12]{->}{E}{H}
\ncarc[arcangle=-12]{->}{E}{H}
%\psset{linecolor=gray}
%\ncarc[arcangle=15]{->}{L}{A}
%\ncarc[arcangle=10]{->}{L}{B}
%\ncarc[arcangle=-5]{->}{L}{C}
%\ncarc[arcangle=-10]{->}{L}{D}
%\ncarc[arcangle=-15]{->}{L}{E}
\end{diagram}
\end{equation}
In this diagram, letters represent systems and arrows represent processes. We call this type of diagram a \emph{finite forest-shaped multigraph}: there are a finite number of connected components, each of which is a finite tree with a root at the top and leaves at the bottom, and we allow the possibility of multiple parallel branches between nodes. This gives us a collection of allowed processes with which to turn leaf systems at the bottom into node systems at the top. We could understand this physically as describing a simple sort of nondeterministic dynamics, where we evolve from an initial system at the bottom of the diagram to a final system towards the top, making a choice of process whenever more than one is available. Assuming for a moment that our systems are composed of sets of states, and our processes are functions, we can ask the following: are there any states of initial systems which will always transform  into the same final state, regardless of the processes chosen? We could also take a more computational perspective, and regard the processes as \textit{constraints}; an analogous question would then be to find the initial states which satisfy these constraints.

The answer to this is provided by the notion of \emph{limit}, an important and widely-used tool in category theory. For our purposes, a limit is described by a system $L$ equipped with a family of processes $l_X:L \to X$, where $X$ ranges over each of the leaf systems in our diagram. These processes must also satisfy a universality property. We can interpret the limit system $L$ as comprising all of the leaf systems in our diagram combined together, but with states identified when they evolve in the same way under the action of the processes in the diagram. We can visualize this with the following diagram:
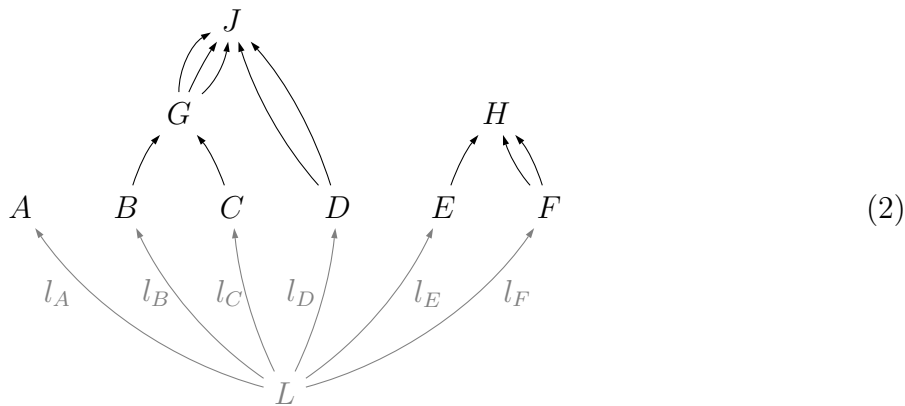
\begin{equation}
\label{diagexample2}
\begin{diagram}[midshaft,width=20pt]
&&&&\rnode{G}{J}&&&&&&
\\
&&&\rnode{F}{G}&&&&&&\rnode{H}{H}&
\\
\rnode{X}{A} &&\rnode{A}{B} & & \rnode{B}{C} && \rnode{C}{D} & \hspace{0pt}&\rnode{D}{E} & & \rnode{E}{F}
\\
\\
&&&&& {\gray\rnode{L}{L}} &&&&&
\ncarc[arcangle=12]{->}{A}{F}
\ncarc[arcangle=30]{->}{F}{G}
\ncarc[arcangle=5]{->}{F}{G}
\ncarc[arcangle=-20]{->}{F}{G}
\ncarc[arcangle=-10]{->}{B}{F}
\ncarc[arcangle=12]{->}{C}{G}
\ncarc[arcangle=-12]{->}{C}{G}
\ncarc[arcangle=10]{->}{D}{H}
\ncarc[arcangle=12]{->}{E}{H}
\ncarc[arcangle=-12]{->}{E}{H}
\psset{linecolor=gray}
{\gray
\ncarc[arcangle=20]{->}{L}{X} \naput[npos=0.79]{l_A\!}
\ncarc[arcangle=15]{->}{L}{A} \naput[npos=0.715]{l_B\!\!\!}
\ncarc[arcangle=10]{->}{L}{B} \naput[npos=0.6]{l_C\!\!}
\ncarc[arcangle=-10]{->}{L}{C} \naput[npos=0.515]{l_D\!}
\ncarc[arcangle=-15]{->}{L}{D} \nbput[npos=0.710]{l_E\!\!}
\ncarc[arcangle=-20]{->}{L}{E} \nbput[npos=0.790]{l_F\!\!\!}
}
\end{diagram}
\end{equation}
The limit system and its associated processes are drawn in gray here.

Our final requirement is that this limit is compatible with the \dag\-functor on our category of processes, which allows us to formally `reverse' processes. Suppose that we compose the processes $l_A ^{}: L \to A$ and $l_A ^\sdag:A \to L$; this composite  evolves a state of $L$ into a state of $A$, and then evolves this back again into a state of $L$. We can consider this as taking a state of $L$ and retaining only that part of it which arises from $A$. This makes sense, as we described the $L$ as being constructed from the combination of all of the leaf systems. It is reasonable to require that a state of $L$ is precisely specified by the sum total of its restrictions to all of the leaf systems. Using our superposition operation, we can express this principle with the following equation:
\begin{equation}
l_A ^{\pdag}; l_A ^\sdag + l_B ^{}; l_B ^\sdag + l_C ^{\pdag}; l_C ^\sdag + l_D ^{} ; l_D ^\sdag + l_E ^{\pdag}; l_E ^\sdag + l _F ^{}; l_F ^\sdag= \id _L ^{}.
\end{equation}
We call this the \emph{normalization condition}. If we can find a limit satisfying this condition then we call it a \emph{\dag\-limit}, and if a category has \dag\-limits for all finite forest-shaped multigraphs then we say that it has \emph{all finite \dag\-limits}.

Categories with all finite \dag\-limits have many interesting properties, which we explore throughout this paper. One useful property is that a category can have at most a single superposition rule (the `+' operation) such that all finite \dag\-limits exist! So the superposition rule is more like a \textit{property} of the \dag\-limits than a structure on the underlying category, and we do not need to specify it explicitly. Having \dag\-limits also implies other useful features, including \emph{nondegeneracy} (or \emph{positivity}) of the \dag\-functor, and \textit{cancellability} for the superposition operation, as we will explore later.

We can now state an interesting result. Suppose we have a category of processes which has a \dag\-functor, compound systems and all finite \dag\-limits, such that the `neutral system'~$I$ --- the unit for constructing compound systems --- is `simple', meaning that the only system smaller than it is the empty system. Then we can show that the analogue of  `quantum amplitudes' in this category take values in an involutive field with characteristic~0, and with orderable fixed field. We interpret this field as analogous to $\mathbb{C}$, the involution as analogous to complex conjugation, and the orderable fixed field as analogous to $\mathbb{R}$.

Furthermore, suppose that the results of measurements in our theory are valued in this orderable fixed field. Then if every bounded sequence of measurement results has a least upper bound, and these least upper bounds are preserved when we add a constant to our measurement results, it follows that our involutive field is $\mathbb{C}$ itself, the orderable fixed field is $\mathbb{R}$, and the order is the familiar order on the real numbers.

An important inspiration for the development of \dag\-limits came from the category \cat{FdHilb}, which has finite-dimensional Hilbert spaces as objects and linear maps as morphisms. This can be considered as a category of processes which emerges from quantum theory. That category has a \dag\-functor, given by taking linear maps to their adjoints, and with this \dag\-functor the category \cat {FdHilb} has all finite \dag\-limits. In fact, as we explore later with Theorem~\ref{innerproducttheorem}, this \dag\-functor can actually be completely defined by its completeness properties. Since knowing the adjoints of the bounded linear maps to a Hilbert space is the same as knowing the inner product on it, this gives a new axiomatization of inner products.

In a nutshell, what we do in this paper is to observe how close the abstract theory of monoidal \dag\-categories comes to describing the structure of real physical theories, and then to `take it seriously'. This is not a new idea. In particular, it has been advanced with much success in the field of quantum computation, especially by Abramsky and Coecke~\cite{ac08-cqm,c06-icp}. We believe that this is a fruitful perspective which holds the promise of delivering significant further results in the future. We note that interesting work has already been carried out which takes the extends the results described here, adding axioms that imply that the resulting category embeds into a category of Hilbert spaces~\cite{h09-eth}.

\subsection*{Acknowledgements}

I am grateful to Samson Abramsky, Chris Isham, Zurab Janelidze, Paul Levy and Paul Taylor, and especially to the anonymous referee, Kevin Buzzard, Chris Heunen and Peter Selinger, for useful comments and discussions.

I have used Paul Taylor's diagrams package, and I am grateful for financial support from the EPSRC and the ONR. I am also grateful to the program committee of \emph{Category Theory 2008} for the opportunity to present some early versions of these results.

\section{\dag\-Functors, \dag\-categories and \dag\-limits}
\label{dagstrucsect}

\subsection*{The \dag\-functor}

Of all the categorical structures that we will make use of, the most fundamental is the \dag\-functor, first made explicit in the context of categorical quantum mechanics by Abramsky and Coecke \cite{ac04-csqp,ac08-cqm}. As described in the introduction, it is motivated by the process of taking the \textit{adjoint} of a linear map between two Hilbert spaces: for any bounded linear map of Hilbert spaces $f:H \to J$, the adjoint $f ^\sdag: J \to H$ is the unique map satisfying
\begin{equation}
\langle f(\phi),\psi \rangle _J = \langle \phi, f ^\sdag (\psi) \rangle _H
\end{equation}
for all $\phi \in H$ and $\psi \in J$, where the angle brackets represent the inner products for each space.

Abstractly, we define a \textit{\dag\-functor} as a contravariant functor from a category to itself, which is the identity on objects, and which satisfies $\dag \circ  \dag = \id _\cat{C}$. A \emph{\dag\-category} is a category equipped with a particular choice of \dag\-functor. These are sometimes known instead as \emph{Hermitian categories} or \emph{$*$-categories}, but we prefer the `\dag{}' notation, since it is snappier and more flexible than `Hermitian', and the symbol `$*$' is also used to denote duals for objects in a monoidal category. Although it is often uninformative to name something after the symbol that denotes it, in our view this is outweighed by the convenience of having a straightforward naming convention \cite{s06-idc} for `\dag\-versions' of many familiar constructions, such as  \dag\-biproducts, \dag\-equalizers, \dag\-kernels, \dag\-limits, \dag\-subobjects and so on, all of which we will encounter below.

The inner product on a Hilbert space is used to calculate the adjoint of a linear map, and in fact the process has a converse \cite{ac05-apt}: knowledge of the adjoints can be used to reconstruct the inner product. To show this, we use the fact that vectors $\phi \in H$ correspond to linear maps $\mathbb{C} \to H$ by considering the image of the number $1$ under any such map. For any two vectors $\phi, \psi \in H$ we can calculate the inner product as
\[
\langle \phi, \psi \rangle _H = \langle \phi(1), \psi(1) \rangle _H = \langle 1, \phi ^\sdag (\psi(1)) \rangle _\mathbb{C} = \phi ^\sdag( \psi (1)),
\]
where the last step follows from the fact that the inner product on the complex numbers is determined by multiplication. For this reason, the \dag\-functor can be thought of not only as an abstraction of the construction of adjoint linear maps, but also as an abstraction of the inner product. However, we note that an \textit{arbitrary} \dag\-functor might give rise to `inner products' which are quite badly-behaved: for example, in a category with zero morphisms, we might have $\langle \phi,\phi \rangle =0$ for $\phi \neq 0$. The \dag\-functors which arise from inner products are characterized in the last section of the paper, in Theorem~\ref{innerproducttheorem}.

We write the action of a \dag\-functor on a morphism $f : A \to B$ as $f ^\sdag : B \to A$, and we refer to the morphism $f ^\sdag$ as the \emph{adjoint} of $f$. We also make the following straightforward definitions, taken from the vocabulary of functional analysis: a morphism is \emph{unitary} if its adjoint is its inverse ($f ; f ^\sdag = \id_A$ and $f ^\sdag; f = \id_B$), an \emph{isometry} if its adjoint is its retraction ($f; f ^\sdag = \id _A$), and is \emph{self-adjoint} if it equals its adjoint ($f=f^\sdag$). If a morphism $f:A \to B$ is an isometry, we also say that $A$ is a \emph{\dag\-subobject} of $B$. If two objects in a \dag\-category have a unitary morphism going between them, we say that they are \emph{unitarily isomorphic}; if every pair of isomorphic objects are unitarily isomorphic, then the \dag\-category is a \emph{unitary \dag\-category}. Many important \dag\-categories are unitary; for example, the \dag\-category of Hilbert spaces with \dag\-functor given by adjoint, the \dag\-category of manifolds and cobordisms with \dag\-functor given by taking the opposite cobordism, or any 2--Hilbert space \cite{b97-hda2}.

There is a natural notion of equivalence between \dag\-categories, which we call \emph{unitary \dag\-equivalence}. Let \cat{C} and \cat{D} be \dag\-categories, with \dag\-functors $\dag:\cat{C} \to \cat{C}$ and $\ddag: \cat D \to \cat D$. These \dag\-categories are unitarily \dag\-equivalent if there exists a functor $F:\cat C \to \cat D$ between them which is part of an adjoint equivalence of categories, such that the unit and counit natural transformations are unitary at every stage, and if it \emph{commutes} with the \dag\-functors, satisfying $F \circ \dag = \ddag \circ F$. As we later show in Lemma~\ref{ude}, a functor can be made a part of a unitary \dag\-equivalence iff it commutes with the \dag\-functors, and is full, faithful, and unitarily essentially surjective.

Merely equipping a category with a \dag\-functor is certainly not trivial, but is perhaps not itself particularly powerful. However, interesting phenomena start to arise when we relate the \dag\-functor to other constructions that we can make with the category. It often pays off to do this enthusiastically, a policy which deserves to be very clearly stated.
\begin{dcp}
\em When working with a \dag\-category, all important structures should be chosen so that they are compatible with the \dag\-functor.
\end{dcp}

\jnoindent
Of course, this is a rule of thumb rather than a technical statement; what counts as an `important structure', and what `compatible' should mean, will depend upon the setting. However, there are many situations in which applying this philosophy bears interesting results:
\begin{itemize}
\item
The constructions made in this paper are a prime example, where we require limits to be compatible with the \dag\-functor.
\item In the study of topological quantum field theories, it is physically well-motivated to require that the functor defining the field theory should be compatible with the \dag\-functor on the category of cobordisms and the \dag\-functor on the category of Hilbert spaces. This gives a \emph{unitary} topological quantum field theory.
\item If a Frobenius algebra in \cat{Hilb} has its multiplication related to its comultiplication by the \dag\-functor, then it is a C*-algebra \cite{cpv08-dfb, v08-cfqa}.
\item When working with a monoidal \dag\-category, it is often useful to require that the left unit, right unit, associativity and braiding isomorphisms are unitary at every stage~\cite{ac05-apt, b97-hda2, v08-cfqa}. \ignore{\item In the categorical framework for the quantum harmonic oscillator described in \cite{v08-cfqho}, a crucial requirement is that the canonical isomorphisms $F(A \oplus B) \simeq F(A) \otimes F(B)$ should be unitary.
}
\end{itemize}
 
\subsection*{Constructing \dag\-limits}

As mentioned earlier, \dag\-limits are the central categorical construction which we will use to prove our results. We use diagrams in the shape of forest-shaped multigraphs with a finite number of leaves, such as example~(\ref{diagexample}) in the introduction. These are defined as diagrams with a finite number of connected components, each of which is a directed tree oriented from a finite number of leaves at the bottom to a root at the top, and for which multiple parallel branches between nodes are allowed. Note that there is no ambiguity about which objects are the leaves; they are exactly the systems in the diagram which are not the target of any process in the diagram (except for an identity process.)

 In the rest of the paper, in the context of \dag\-limits, we will often simply  refer to these forest-shaped multigraphs as \textit{diagrams}. We say that such a diagram is finite when it has a finite number of arrows. On page \pageref{moregeneral} we give a generalized definition of \dag\-limit which applies to a much larger class of diagrams, but \dag\-limits of finite forest-shaped multigraphs are sufficient to obtain our results.

Let $F: \cat J \to \cat C$ define such a diagram in the category \cat C. A \emph{cone} for this diagram is an object $X$ in \cat C, equipped with \emph{cone maps} $x_S:X \to F(S)$ for all objects $S$ of \cat J, such that for any map $f:A \to B$ in \cat J the equation $x_A ; F(f) = x_B$ holds. A \emph{limit} for the diagram is a special cone $L$, equipped with cone maps $l_S:L \to F(S)$, such that for any cone $(X,x_S)$ for the diagram, there is a unique map $m:X \to L$ such that $x_S = m; l_S$ for all objects $S$ in \cat J. For more information about limits in category theory have a look at any introductory category theory textbook, such as \cite{m95-ecet}.

If \cat C is a \dag\-category with a superposition rule `+' on the hom-sets --- or technically, which is \emph{enriched in commutative monoids} --- then 
a \emph{\dag\-limit} for a diagram $F:\cat J \to \cat C$ is a limit for the diagram in the usual sense, such that the normalization condition
\begin{equation}
\label{norm}
\sum _{S} l_S ^\pdag; l_S ^\sdag = \id_L, \qquad \textrm{$S$ is a leaf in \cat J}
\end{equation}
holds, where the maps $l_S : L \to F(S)$ are the projection maps from the limit object to the leaves. Since we require all diagrams to only have finite number of leaves, this is a finite sum, and will always be well-defined. It seems that this definition is sensitive to the definition of the superposition rule `+' used to define the summation, but in fact it is not, as explained by Lemma~\ref{sumlemma}. If a \dag\-category has a \dag\-limit for every diagram then we say it has \emph{all \dag\-limits}. If it only has a \dag\-limit for all finite diagrams, then we say that it has all \emph{finite \dag\-limits}.

These \dag\-limits are, in particular, ordinary limits, and so will be isomorphic to any other ordinary limit. \label{limitdiscussion} However, between themselves, \dag\-limits satisfy a stronger universal property~--- they are unique up to unique \emph{unitary} isomorphism.

\newpage
\begin{lemma}
\label{uniqueness}
In a \dag\-category, any \dag\-limit is unique up to unique unitary isomorphism.
\end{lemma}

\jbeginproof
Let $F : \cat J \to \cat C$ be a diagram, and let $(L,l_S)$ and $(M,m_S)$ be \dag\-limits for this diagram, where $l_S: L \to F(S)$ and $m_S:M \to F(S)$ are the respective limit maps, and $S$ is a variable that ranges over the leaf objects of $\cat J$. Then by the properties of limits, there must be a unique comparison isomorphism $c:L \to M$ with the property that $c; m_S = l_S$ for all $S$. By the normalization condition we have the equations $\sum _{S} l_S ^\pdag; l_S ^\sdag = \id _L$ and $\sum _{S} m_S ^\pdag; m_S ^\sdag = \id _M$, and we employ these in the following way to show that $c; c ^\dag = \id_L$:
\begin{equation*}
c; c ^\sdag = c; \id _M; c ^\sdag =  c;
\shiftbrackets
{3.3pt}
{$\displaystyle \sum _{S} m_S ^\pdag; m_S ^\sdag \!$}
\hspace{-1pt} ; c ^\sdag =\sum _{S} c; m_S ^\pdag; m_S ^\sdag; c ^\sdag = \sum _{S} l_S ^\pdag; l_S ^\sdag = \id_L.
\end{equation*}
It can be shown in a similar way that $c ^\sdag; c = \id_M$, and so $c$ is unitary.
\end{proof}

\subsubsection*{\dag\-Products and \dag\-equalizers}

We will make substantial use of two particularly important types of \dag\-limit. The first type of \dag\-limit is a \emph{finite \dag\-product}, which is the \dag\-limit of a finite discrete diagram, for which every object is  a leaf. The second type is a \emph{finite \dag\-equalizer}, which is the limit of a diagram consisting of a finite number of arrows, all of which have the same source object and the same target object; this has exactly one leaf vertex. We can draw these \dag\-limits as follows, with the diagram in black and the \dag\-limit and its associated maps in grey:
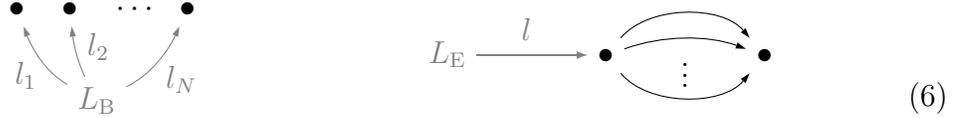
\begin{gather}
\begin{array}{c@{\hspace{40pt}}c}
\begin{diagram}[width=10pt]
\rnode{a}{\bullet} && \rnode{b}{\bullet} && \cdots && \rnode{c}{\bullet}
\\
&&&\rnode{l1}{\gray L_{\mathrm{B}}}&
\end{diagram}
&
\begin{diagram}[width=30pt]
\rnode{l2}{\gray L_{\mathrm{E}}} &\gray\rTo ^l& \black \rnode{d}{\bullet} && \rnode{e}{\bullet}
\end{diagram}
\ncarc[arcangle=50]{->}{d}{e}
\ncarc[arcangle=20]{->}{d}{e} \Bput{\textrm{\raisebox{50pt}{$\black\vdots$}}}
\ncarc[arcangle=-50]{->}{d}{e}
\psset{linecolor=gray}
{\gray
\ncarc[arcangle=20]{->}{l1}{a} \Aput{l_1\!}
\ncarc[arcangle=10]{->}{l1}{b} \Bput{l_2}
\ncarc[arcangle=-20]{->}{l1}{c} \Bput{l_N}
}
\\[30pt]
l_1 ^{} ; l_1 ^\sdag + l_2 ^{} ; l_2 ^\sdag + \cdots + l_N ^{} ; l_N ^\sdag = \id _{L_{\mathrm{B}}}
&
l; l^\sdag = \id _{L_{\mathrm{E}}}
\end{array}
\end{gather}
The relevant form of the normalization condition (\ref{norm}) is given underneath each diagram.

We emphasize that a \dag\-equalizer is exactly a conventional category-theoretical equalizer, such that the equalizing map is an isometry. This extra isometry condition is a natural one to consider in a \dag\-category, since equalizers are always monic, and the isometry condition can be considered as a strengthening of the monic property. We also define a \emph{\dag\-kernel} to be a \dag\-equalizer of a parallel pair consisting of an arrow and the zero arrow. This research programme was born out of a study of the properties of  \dag\-categories with \dag\-equalizers, and I am grateful to Peter Selinger for suggesting them as a construction.

A first useful result is that \dag\-products are exactly  \emph{\dag\-biproducts}, which are well-known generalizations of the concept of `orthogonal direct sum': for any two objects $A$ and $B$, their \dag\-biproduct is an object $A \oplus B$ equipped with \emph{injection morphisms} $i_A : A \to A \oplus B$ and $i_B:B \to A \oplus B$ satisfying the following equations:
\newcommand\tempgap{\hspace{108pt}}
\newcommand\tx{4.5pt}
\begin{gather}
\nonumber
i_A ^\sdag; i_A ^\pdag + i_B ^\sdag; i_B ^\pdag = \id_{A \oplus B}
\\
\hspace{-\tx}
i_A ^\pdag; i_A ^\sdag = \id _A
\hspace{\tx}
\tempgap
i_B ^\pdag; i_B ^\sdag = \id _B
\\
\nonumber
i_A ^\pdag; i_B ^\sdag = 0 _{A,B}
\tempgap
i_B ^\pdag; i_A ^\sdag = 0 _{B,A}
\end{gather}
The adjoints to the injection morphisms are called the \emph{projection morphisms}. This definition of \dag\-biproduct  generalizes in an obvious way to any finite list of objects.
\begin{lemma}
\label{dagbiproducts}
The \dag\-limit of a discrete diagram (that is, a \dag\-product) is the \dag\-biproduct of the objects of the diagram, and the cone maps are the \dag\-biproduct projections.
\end{lemma}
\jbeginproof
We prove our lemma for the case of a discrete diagram with two objects; the extension to any finite discrete diagram of objects is straightforward. Consider the \dag\-limit of the diagram consisting of two objects, $A$ and $B$. The \dag\-limit is a limit object $L$, equipped with morphisms $l_A$ and $l_B$ which satisfy
\begin{equation}
\label{db1}
l_A ^\pdag; l_A ^{\sdag} + l_B ^\pdag ; l_B ^\sdag = \id _{L}.
\end{equation}
Since $L$ is the limit, there is a unique map $\langle 0_{B,A},\id_{B} \rangle : B \to L$ with $\langle 0_{B,A},\id_{B} \rangle; l_A = 0_{B,A}$ and $\langle 0_{B,A},\id_{B} \rangle ; l_B = \id_B$, where $0_{B,A}:B \to A$ is the unit for the enrichment in commutative monoids. Precomposing (\ref{db1}) with this map we obtain $l_B ^\sdag = \langle 0_{B,A},\id_B \rangle$, and so we have
\begin{align}
l_B ^\sdag; l_A ^\pdag &= 0_{B,A},
&
l_B ^\sdag; l_B ^\pdag &= \id _{B}.
\end{align}
Similarly we can show that $l_A ^\sdag = \langle \id _A, 0 _{A,B} \rangle: A \to L$, which leads to the equations
\begin{align}
l_A ^\sdag; l_B ^\pdag &= 0_{A,B},
&
l_A ^\sdag; l_A ^\pdag &= \id _{A}.
\end{align}
Altogether, these equations witness the fact that $L$\ is the \dag\-biproduct of $A$ and $B$, with projections $l_A^{}$, $l_B^{}$ and injections $l_A ^\sdag$, $l_B ^\sdag$.
\end{proof}

In a category with biproducts there is a unique enrichment in commutative monoids, which can be defined in the following way:
\begin{equation}
\label{cmonenrichment}
\begin{diagram}[height=30pt,width=45pt]
A & \rTo^{f+g} & B
\\
\dTo <{\Delta _A}
&&
\uTo >{\nabla_B}
\\
A \oplus A
& \rTo _{f \oplus g} &
B \oplus B
\end{diagram}
\end{equation}
Here, the \emph{diagonal map} $\Delta_A : A \to A \oplus A$ is the unique map having the property that $\Delta_A; i_1 {}^\dag = \Delta_A; i_2 {}^\sdag = \id_A$, where $i_1{}^ \sdag, i_2 ^\sdag:A \oplus A \to A$ are the projections onto the first and second component of the biproduct respectively. The \emph{codiagonal} \mbox{$\nabla_B : B \oplus B \to B$} is defined in a similar way as the unique map satisfying \mbox{$i_1; \nabla _B = i_2 ; \nabla _B = \id _B$}. It is straightforward to show that the biproduct operation on morphisms satisfies \mbox{$(f \oplus g) ^\sdag = f ^\sdag \oplus g ^\sdag$} for every pair of morphisms $f$ and $g$. Also, we have \mbox{$f;(g+h)=(f;g)+(f;h)$} and $(g+h);j = (g;j)+ (h;j)$ for all morphisms $f,g,h,j$ of the correct types, as can be directly checked by applying equation~\eqref{cmonenrichment}.\label{adddist}

The diagonal $\Delta_A$ and the codiagonal $\nabla_A$ are adjoint to each other, as demonstrated by the following lemma.

\newpage
\begin{lemma}
\label{daggerdcod}
For any \dag\-biproduct $A \oplus A$, the diagonal $\Delta_A:A \to A \oplus A$ and codiagonal $\nabla_A : A \oplus A \to A$ satisfy $\Delta _A {}^\sdag = \nabla _A$.
\end{lemma}

\jbeginproof
We see that $\id _A = (\id _A) ^\sdag = (\Delta _A; p_i) ^\sdag = p_i {}^\dag ; \Delta _A {}^\sdag$, where $i \in \{1,2\}$ and $p_i$ is a projector onto one of the factors of the biproduct. But $\id _A = p_i {}^\dag ; \nabla _A$ for all $i$ is the defining equation for the codiagonal, and so $\Delta_A {} ^\sdag = \nabla_A$.
\end{proof}

\jnoindent
From this lemma, and from the definition of $f+g$ given by equation \eqref{cmonenrichment}, it follows that the commutative monoid structure is compatible with the action of the \dag\-functor, satisfying
\begin{equation}
\label{dagadd}
(f+g) ^\sdag = f ^\sdag + g ^\sdag
\end{equation}
for all parallel morphisms $f$ and $g$.

In a category with biproducts we have a matrix calculus available to us: a morphism $f:{\bigoplus}_i A_i \to {\bigoplus}_i B_i$ corresponds to a matrix of morphisms $f_{i,j}:A_i \to B_j$, and composition of morphisms is given by matrix multiplication. In any \dag\-category with \dag\-biproducts, it can be shown that the adjoint of a matrix has the following form:
\begin{equation}
\left(
\begin{array}{cccc}
f & g & \cdots & x
\\
h & j
\\
\vdots & & \ddots
\\
y &&& z
\end{array}
\right) ^{ \displaystyle \dag}
=\,\,
\left(
\begin{array}{cccc}
f^\sdag & h^\sdag & \cdots & y^\sdag
\\
g^\sdag & j^\sdag
\\
\vdots & & \ddots
\\
x^\sdag &&& z^\sdag
\end{array}
\right)
\end{equation}
This is just the familiar matrix conjugate-transpose operation, with the `conjugate' of each entry in the matrix being its adjoint.

The category \cat{Hilb} has all finite \dag\-limits, and so in particular has both \dag\-biproducts and \dag\-equalizers: the \dag\-biproduct of a finite list of Hilbert spaces is given by their direct sum, and for some parallel set of linear maps $A \to B$, their \dag\-equalizer is given by an isometry with image equal to the largest subspace of $A$ on which all the linear maps agree.

\subsubsection*{Uniqueness of the superposition rule}

Because of the normalization condition~\eqref{norm} it seems that the definition of \dag\-limits depends on the choice of the superposition rule `+', which we refer to as the \emph{enrichment in commutative monoids}. This is true, but can be easily overcome, thanks to the following fact: if by some enrichment in commutative monoids a \dag\-category at least has \dag\-limits of discrete diagrams and of the empty diagram, then the category in fact admits a \textit{unique} enrichment in commutative monoids. So in particular, if a \dag\-category admits an enrichment in commutative monoids such that it has all finite \dag\-limits, then that enrichment is determined uniquely. This can be shown by considering Lemma~\ref{dagbiproducts} along with the following well-known result.

\begin{lemma}
\label{sumlemma}
Suppose that a category has a zero object and all finite biproducts. Then it has a unique enrichment in commutative monoids.
\end{lemma}

\jbeginproof
For any hom-set $\Hom(A,B)$, write $0_{A,B}:A \to B$ for the unique morphism which factors through the zero object, and $\widetilde 0 _{A,B} : A \to B$ for the unit morphism encoded by  the enrichment in commutative monoids. Clearly  $0_{A,0}= \widetilde 0 _{A,0}$ and $0 _{0,A} = \widetilde 0 _{0,A}$, since those hom-sets only contain a single element. Using the axiom that $\widetilde 0_{A,B};f= \widetilde 0 _{A,C}$ for all objects $C$ and all morphisms $f:B \to C$, we obtain $0_{A,B} = 0 _{A,0}; 0_{0,B} = \widetilde 0 _{A,0};0_{0,B}= \widetilde 0 _{A,B}$, and so the zero morphisms and the unit morphisms for the enrichment coincide. As a result, for the rest of this proof, we will use $0_{A,B}$ to represent both the zero and unit morphisms.

In a category with biproducts, for any $f,g:A \to B$, we can define a morphism $f \boxplus g:A \to B$ as
\begin{equation}
\begin{diagram}[midshaft,width=40pt]
A & \rTo^ {\Delta_A} & A \times A & \rTo ^{\alpha ^{-1}} & A + A & \rTo ^{(f \,\,\,\, g)} & B,
\end{diagram}
\label{fpg}
\end{equation}
where $(f\,\,\,g)$ is the unique map with $i_1;(f\,\,\,\,g)=f$ and $i_2;(f\,\,\,\,g)=g$. The map $\alpha ^{-1}$ is the inverse of the map  $\alpha: A + A \to A \times A$, which is the unique map such that:
\renewcommand\tempgap{\hspace{30pt}}
\begin{equation}
i_1;\alpha;p_1 = \id _A
\tempgap
i_1 ; \alpha; p_2 = 0_{A,A}
\tempgap
i_2; \alpha; p_1 = 0_{A,A}
\tempgap
i_2 ; \alpha; p_2 =\id_A
\end{equation}
Here $i_1$ and $i_2$ are the coproduct injections into $A +A$, and $p_1$ and $p_2$ are the product projections out of $A \times A$. We will demonstrate that $\alpha ^{-1} = p_1;i_1+p_2;i_2$. Consider $\alpha; \alpha ^{-1}= \alpha; p_1; i_1 + \alpha; p_2; i_2$. Then $i_1; \alpha; \alpha ^{-1} = i_1$ and $i_2; \alpha; \alpha^{-1} = i_2$, and by the universal property satisfied by coproduct injections, we must have $\alpha; \alpha^{-1} = \id_A$. We can show similarly that $\alpha^{-1}; \alpha = \id_A$, and so $\alpha^{-1}$ and $\alpha$ are inverse. Substituting our expression for $\alpha ^{-1}$ into equation~\eqref{fpg}, we obtain
\begin{align}
\nonumber
f \boxplus g &= \Delta_A; (p_1; i_1 + p_2;i_2); (f\,\,\,\,g) 
\\
&= 
\nonumber
\Delta_A; p_1; i_1; (f\,\,\,\,g) + \Delta_A; p_2;i_2; (f\,\,\,\,g) 
\\
&= \id_A; f + \id_A; g = f + g.
\end{align}
But $f \boxplus g$ was defined without reference to the \cat{CMon}-enrichment operation `${+}$', and so it follows that this is the only enrichment that can exist.
\end{proof}

\subsection*{Properties of \dag\-categories with \dag\-limits}

The existence of all finite \dag\-limits in a \dag\-category guarantees some interesting properties. As a general rule of thumb, these properties are those which are familiar from the category of complex Hilbert spaces.

\subsubsection*{Nondegeneracy}
The first property we will examine is \emph{nondegeneracy}, also called \emph{positivity} by some authors~\mbox{\cite[Definition~8.9]{hm06-aqft}}. In a \dag\-category with a zero object, we define the \dag\-functor to be nondegenerate if $f ; f ^\sdag = 0$ implies $f=0$ for all morphisms $f$. We show now that this property is closely linked to the existence of \dag\-equalizers.

\newpage
\begin{lemma}[Nondegeneracy]
\label{nondegen}
In a \dag\-category with a zero object and finite \dag\-equalizers, the \dag\-functor is nondegenerate.
\end{lemma}
\jbeginproof

Let $f:A \to B$ be an arbitrary morphism satisfying $f ; f ^\sdag = 0 _{A,A}$. Then $f$ must factor through the \dag\-kernel of $f ^\sdag$ as indicated by the following commuting diagram, where the factorising morphism is denoted $\tilde{f}$, and $(K,k)$ forms the \dag\-kernel of $f ^\sdag$:
\begin{equation}
\begin{diagram}[nohug,midshaft,width=20pt,height=35pt]
&
&\rnode{1}{A}&
&&&
\\
\rnode{2}{K} && \rInto^{k} && \rnode{3}{B} & \hspace{35pt} \pile{ \rTo ^{f^\sdag} \\ \rTo _{0 _{B,A}}} & A
\nccurve[angleA=180,angleB=80]{->}12\Bput {\tilde f}
\nccurve[angleA=0,angleB=100]{->}13 \Aput f
\end{diagram}
\vspace{-5pt}
\end{equation}
By definition we have $k;f ^\sdag = 0 _{K,A}$, and we apply the \dag\-functor to obtain $f; k ^\sdag = 0 _{A,K}$. Also, since $(K,k)$ is a \dag\-kernel, $k$ is an isometry, which means $k ; k^\sdag = \id _{K}$. We can now demonstrate that $f$ is zero:
\[
f = \tilde{f}  ; k = \tilde{f} ; k ; k ^\sdag ; k = f; k^\sdag; k = 0 _{A,K}; k = 0 _{A,B}. \qedhere
\]
\end{proof}

\jnoindent
An important feature of this proof, which will recur in other proofs throughout this paper, is that although the \dag\-functor is used sparingly, it is used crucially: in this case, to translate $k; f ^\sdag = 0 _{K,A}$ into $f ; k ^\sdag = 0 _{A,K}$.

The category of complex Hilbert spaces has finite \dag\-equalizers, and so this lemma can be seen as `explaining' why that category has a nondegenerate \dag\-functor. Conventionally, the nondegeneracy property in \cat{Hilb} would instead be proved using the fact that inner products on Hilbert spaces are necessarily \emph{positive definite}. In this way, it is clear that there is some connection between positive-definiteness of inner products and the existence of \dag\-equalizers; we formalize this later with Theorem~\ref{innerproducttheorem}, which demonstrates that in a \dag\-category with \dag\-equalizers, each object is endowed with a canonical notion of inner product.

\subsubsection*{Cancellability}

We now study various cancellability properties satisfied by the additive structure on the hom-sets. Say that a commutative monoid is \emph{cancellable} if, for any three elements $a,b,c$ in the monoid, $a+c=b+c \, \Rightarrow \, a=b$. We are motivated to study this condition since, in particular, it is satisfied by the addition of linear maps between Hilbert spaces. We now show that it follows as a consequence of having \dag\-limits.

\begin{lemma}[Cancellable addition]
\label{additivelemma}
In a \dag\-category with all finite \dag\-limits, hom-set addition is cancellable; that is, for arbitrary $f,\,g,\,h$ in the same hom-set,
\vspace{-\jskip}
\begin{equation}
f+h=g+h \,\Rightarrow\, f=g.
\end{equation}
\end{lemma}

\jbeginproof
Let $f,\, g,\, h:A \to B$ be morphisms satisfying the equation $f+h = g+h$. Then we can form the following commuting diagram, consisting of a \dag\-equalizer $(E,e)$ for the parallel pair $(f \,\,\, h)$ and $(g \,\,\, h)$ along with two cones $(A,i_2)$ and $(A,\Delta_A)$:
\begin{equation}
\begin{diagram}[nohug,midshaft,width=30pt,height=40pt]
&
\mathmove{5pt}{-25pt}{\tilde i_2}
&\rnode{1}{A}&
\mathmove{5pt}{13pt}{i_2 = \big(
\hspace{-2pt}
{
\begin{smallmatrix} 0_{A,A} \\[2pt] \id_A \end{smallmatrix}
}
\hspace{-2pt} \big)
\hspace{-43pt}
}
&&&
\\
\rnode{2}{E} && \rInto^{e= \big(
\hspace{-2pt}
{
\begin{smallmatrix} e_1 \\[2pt] e_2 \end{smallmatrix}
}
\hspace{-2pt} \big)} && \rnode{3}{A \oplus A} & \hspace{35pt} \pile{ \rTo ^{(f \,\,\, h)} \\ \rTo _{(g \,\,\, h)}} & B
\\
& \mathmove{-3pt}{-25pt}{\widetilde {\Delta} _A} & \rnode{4}{A} & \mathmove{-3pt}{13pt}{\Delta_A = \big(
\hspace{-2pt}
{
\begin{smallmatrix} \id_A \\[2pt] \id_A \end{smallmatrix}
}
\hspace{-2pt} \big)
\hspace{-43pt}
}
 &&&
\nccurve[angleA=180,angleB=80]{->}12
\nccurve[angleA=0,angleB=100]{->}13
\nccurve[angleA=180,angleB=-80]{->}42
\nccurve[angleA=0,angleB=-100]{->}43
\end{diagram}
\end{equation}
The morphism $i_2$ is the injection of the second factor into the \dag\-biproduct, and the morphism $\Delta_A$ is the diagonal for the \dag\-biproduct. Since $i_2$ and $\Delta_A$ are cones they must factorize uniquely through $e$, and we denote these factorizations by $\tilde i _2$ and $\widetilde \Delta _A$ respectively. The condition that $e$ is an isometry gives the equation
\begin{equation}
\label{eisom}
e_1 ^\pdag; e_1^\sdag + e^\pdag_2; e_2 ^\sdag = \id _E.
\end{equation}
Precomposing with $\tilde i _2$ gives $e_2 ^\sdag = \tilde i _2$, and postcomposing this with $e_1$ and $e_2$ respectively gives
\begin{align}
\label{21iszero}
e_2 ^\sdag ; e^\pdag_1 &= 0_{A,A},
\\
\label{22isone}
e_2 ^\sdag; e^\pdag_2 &= \id_A.
\end{align}
Similarly, precomposing (\ref{eisom}) with $\widetilde \Delta _A$ gives us $e_1 ^\sdag + e_2 ^\sdag = \widetilde \Delta _A$, and postcomposing with  with $e_1$ and $e_2$ respectively gives
\begin{align}
\label{1121isone}
e_1^\sdag ; e^\pdag_1 + e_2 ^\sdag; e^\pdag_1 &= \id _A,
\\
\label{1222isone}
e_1 ^\sdag ; e^\pdag_2 + e_2 ^\sdag ; e^\pdag_2 &= \id _A.
\end{align}

\jnoindent
We will show that $i_1 =\jbigl \begin{smallmatrix} \id_A \\ 0_{A,A} \end{smallmatrix}\jbigr:A \to A \oplus A$ is a cone for the parallel pair, which directly leads to the required conclusion $f=g$. We must find a factorising morphism $c:A \to E$ which gives $i_1$  upon composition with $e:E \to A \oplus A$. We choose $c=e_1 ^\sdag$, and so we must show that $e_1 ^\sdag; e^\pdag_1 = \id _A$ and $e_1 ^\sdag; e^\pdag_2 = 0 _{A,A}$. The first of these is obtained by applying equation (\ref{21iszero}) to equation (\ref{1121isone}), and the second by applying the $\dag$-functor to equation~(\ref{21iszero}).
\end{proof}

\jnoindent
An important observation is that it seems to be impossible to avoid the use of the \dag\-functor for the final stage of this proof. Without it, the strongest equation that we can easily derive for the endomorphism $e_1^\sdag ; e^\pdag_2$ is
\begin{equation}
e_1 ^\sdag; e^\pdag_2 + \id _A = \id _A,
\end{equation}
obtained by combining equations (\ref{22isone}) and (\ref{1222isone}). Of course, without the cancellability property that we are trying to prove, this is not enough to establish that $e_1 ^\sdag; e_2 = 0_{A,A}$.

One use for this lemma is to demonstrate that a particular category does not have all finite \dag\-limits, which is usually more difficult than checking whether hom-set addition is cancellable. For example, the category \cat{Rel} of sets and relations is a \dag\-category with \dag\-functor given by relational converse, and it has finite \dag\-biproducts. Since $\id_1 + \id _1 = \id _1$ in this category, it does not have cancellable addition, and so by the theorem does not have all  \dag\-limits. (Of course, since \cat{Rel} does not even have equalizers, this is not surprising.)

We now investigate another form of cancellability. In a category enriched in commutative monoids, for any natural number $n$ and any morphism $f$, we define the \emph{$n$-fold sum} of $f$ to be $n \cdot f := f+f+\cdots + f$, where we sum over a total of $n$ copies of~$f$. We can then prove the following lemma.

\begin{lemma}
\label{addlemma}
In a \dag\-category with all finite \dag\-limits, for any $f,g$ in the same hom-set, if there exists a nonzero $n$ with $n \cdot f = n \cdot g$, then $f = g$.
\end{lemma}

\jbeginproof
Consider the following commutative diagram, where $f,g:A \to B$ are morphisms satisfying $n \cdot f = n \cdot g$:
\begin{equation}
\begin{diagram}[nohug,midshaft]
\rnode{E}{E}
\\
 \uTo< {\widetilde \Delta}& 
\rnode{N}{A ^{\oplus n}} & \pile{\rTo^{(f\,\,f\,\,\cdots\,\,f)} \\ \rTo _{ (g\,\, g\,\, \cdots \,\,g)}} & B
\\
\rnode{A}{A} &&&
\nccurve[angleA=0,angleB=90]{->}EN\Aput{e}
\nccurve[angleA=0,angleB=-90]{->}AN\Bput{\Delta}
\end{diagram}
\end{equation}
The diagonal morphism $\Delta : A \to A ^{\oplus n}$ is a cone for the parallel pair, and so it factors uniquely through the \dag\-equalizer $e:E {\to}{A ^{\oplus n}}$ as $\widetilde \Delta : A {\to} E$. Let $p_i: A ^{\oplus n} \to A$ be the projection onto the $i$th factor of the \dag\-biproduct, and define \mbox{$e_i := e;p_i:E \to A$} as the $i$th element of the \dag\-equalizer morphism $e:E \to A ^{\oplus n}$.
We have \mbox{$\Delta = \widetilde \Delta ;e = \widetilde \Delta ;e ;e ^\sdag; e = \Delta ; e ^\sdag ; e$}, and by postcomposing with $p_1$ we obtain $\id _A = \sum _{i \in N}e ^\sdag _i;e ^\pdag_1$ where $N$ is a set with $n$ elements. Taking the adjoint of this gives
$\id_A = \sum _{i \in N} e_1 ^\sdag; e _i ^\pdag$. Since $e$ is a cone we have $\sum _{i \in N} (e_i;f) = \sum _{i \in N} (e_i;g)$, and by precomposing with $e_1 ^\sdag$ and reorganising we obtain $(\sum _{i \in N} e_1 ^\sdag; e_i ^\pdag);f = (\sum _{i \in N} e_1 ^\sdag ; e_i ^\pdag);g$. We have already shown that $\sum _{i \in N} e_1 ^\sdag; e _i ^\pdag = \id_A$, and so we obtain $f=g$.
\end{proof}

Finally we show that the $n$-fold sum operation has an inverse for any positive $n$. It follows from this that we can construct \textit{fractions} of morphisms.

\begin{lemma}
\label{fraclemma}
In a \dag\-category with all finite \dag\-limits, for each object $A$ and each nonzero natural number $n$, there exists a unique morphism $\frac{\id_A}{n}: A \to A$ with $n \cdot \frac{\id_A}{n} = \id_A$.
\end{lemma}
\jbeginproof
Consider the equalizer diagram consisting of the projection maps $p_i: A ^{\oplus n} \to A$. Let $e:E \to A ^{\oplus n}$ be their \dag\-equalizer, and let $\Delta : A \to A ^{\oplus n}$ be the $n$-fold diagonal map, which is also an equalizer. Then there is a unique map $\widetilde \Delta:A \to E$ mediating between these equalizers.
\begin{equation}
\begin{diagram}[midshaft,nohug]
\rnode{E}{E} &&&
\\
 &  \rnode{N}{A ^{\oplus n}} & \pile{\scriptsize \rTo ^{p_1} \\ \cdots\hspace{-5pt} \\ \scriptsize\rTo _ {p_n}} & A
\\
\rnode{A}{A} &&&
\nccurve[angleA=90,angleB=-90]{->}AE\Aput {\widetilde \Delta }
%\nccurve[angleA=-180,angleB=180,ncurv=1.0,linestyle=dashed]{->}EA\Bput{u}
\nccurve[angleA=0,angleB=90]{->}EN\Aput{e}
\nccurve[angleA=0,angleB=-90]{->}AN\Bput{\Delta }
\end{diagram}
\end{equation}
Let $e_i:E \to A$ be the $i$th component of the \dag\-equalizer $e$, defined by $e_i= e; p_i$. Since $e$ is an equalizer for the morphisms $p_i$, each of these components $e_i$ are equal. Then $\id_A = \Delta ; p_1 = \widetilde \Delta ; e_1 = \widetilde \Delta ; e; e ^\sdag; e_1 = \Delta ; e ^\sdag; e_1 = \sum _{i} e_i ^\sdag; e_1 ^\pdag = \sum_i e_1 ^\sdag; e_1 ^\pdag = n \cdot e_1 ^\sdag; e_1 ^\pdag$, and we can define $\frac{\id_A}{n} := e_1 ^\sdag; e_1 ^\pdag$. It follows from Lemma~\ref{addlemma} that this morphism is the unique one with the necessary property.
\end{proof}

\subsubsection*{Exchange lemma}

The final property that we prove is an `exchange lemma', which identifies a restriction on the algebra of morphism composition in the presence of \dag\-limits. It can be seen as a stronger form of the nondegeneracy property demonstrated in Lemma~\ref{nondegen}. We will use this exchange lemma in an essential way in the next section, to prove that our generalized real numbers admit a total order.

\begin{lemma}[Exchange]
\label{exchangelemma}
In a \dag\-category with all finite \dag\-limits, for any parallel morphisms $f$ and $g$,
\begin{equation}
f^\sdag; f + g ^\sdag; g = f ^\sdag; g + g ^\sdag; f \, \Rightarrow \, f = g.
\end{equation}
\end{lemma}

\begin{proof}
%\jbeginproof
Let $f,g:A \to B$ be morphisms satisfying $f^\sdag; f + g ^\sdag; g = f ^\sdag; g + g ^\sdag; f$.  As might be expected from the earlier lemmas, our proof strategy is to construct a \dag\-equalizer diagram, which in this case consists of the parallel pair $(f\,\,\,g)$ and $(g\,\,\,f)$. We next deduce the existence of certain cones, $(B,p)$ and $(B,q)$, which factorize through the \dag\-equalizer $(E,e)$ via $\tilde p$ and $\tilde q$ respectively:
\begin{equation}
\begin{diagram}[nohug,midshaft,width=30pt,height=40pt]
&
\mathmove{5pt}{-25pt}{\tilde p}
&\rnode{1}{B}&
\mathmove{5pt}{13pt}{p = \left(
\hspace{-2pt}
{
\begin{smallmatrix} f^\sdag \\[2pt] g^\sdag \end{smallmatrix}
}
\hspace{-2pt} \right)
\hspace{-43pt}
}
&&&
\\
\rnode{2}{E} && \rInto^{e= \big(
\hspace{-2pt}
{
\begin{smallmatrix} e_1 \\[2pt] e_2 \end{smallmatrix}
}
\hspace{-2pt} \big)} && \rnode{3}{A \oplus A} & \hspace{35pt} \pile{ \rTo ^{(f \,\,\, g)} \\ \rTo _{(g \,\,\, f)}} & B
\\
& \mathmove{-3pt}{-25pt}{\tilde q} & \rnode{4}{B} & \mathmove{-3pt}{13pt}{q = \left(
\hspace{-2pt}
{
\begin{smallmatrix} g ^\sdag \\[2pt] f ^\sdag \end{smallmatrix}
}
\hspace{-2pt} \right)
\hspace{-43pt}
}
 &&&
\nccurve[angleA=180,angleB=80]{->}12
\nccurve[angleA=0,angleB=100]{->}13
\nccurve[angleA=180,angleB=-80]{->}42
\nccurve[angleA=0,angleB=-100]{->}43
\end{diagram}
\end{equation}
Since $e:E \to A \oplus A$ is a \dag\-equalizer we have $p ^\sdag = e ^\sdag; {\tilde p} ^\sdag = e ^\sdag; e; e ^\sdag; \tilde p ^\sdag = e^\sdag; e; p^\sdag$, and similarly $q ^\sdag = e^\sdag; e; q^\sdag$. The equalising morphism $e$ is a cone, and given that $(f\,\,\,g) = p ^\sdag$ and $(g\,\,\,f) = q ^\sdag$, we obtain $e; p ^\sdag = e ; q ^\sdag$. It is then straightforward to see that $(f\,\,\,g) = p ^\sdag = e ^\sdag; e; p ^\sdag = e ^\sdag; e; q ^\sdag = q ^\sdag = (g\,\,\,f)$, and so $f=g$ as required.
\end{proof}

\jnoindent
We call this the `exchange lemma' since, passing from one side of the main equation to the other, the morphisms $f$ and $g$ exchange positions. Many interesting relations arise as special cases of this lemma. Choosing $g=0_{A,B}$ we obtain the nondegeneracy result of Lemma~\ref{nondegen},
\begin{equation}
f^\sdag ; f = 0 _{B,B} \, \Rightarrow \, f = 0_{A,B},
\nonumber
%\eqno{(\ref{nondegeneq})}
\end{equation}
so the exchange lemma can be seen as a generalization of this. Another interesting special case is $g = \id _{A,A}$£, which gives, for all $f:A \to A$,
\begin{equation}
\id _A + f^\sdag ; f = f + f ^\sdag \, \Rightarrow \, f = \id _A.
\end{equation}
Finally, choosing $f$ and $g$ to be endomorphisms and $f = g ^\sdag$, we obtain
\begin{equation}
f; f ^\sdag + f ^\sdag; f = f; f + f ^\sdag; f ^\sdag \, \Rightarrow \, f = f ^\sdag,
\end{equation}
which gives a new way to identify self-adjoint endomorphisms.

Of course, since \cat{Hilb} is our primordial example of a \dag\-category with all finite \dag\-limits, the exchange lemma and its corollaries holds there. However, in this category --- or in any \dag\-category for which hom-set addition is invertible --- the exchange lemma is equivalent to the nondegeneracy condition, by moving terms across the equality and factorizing:
\begin{align*}
 f^\sdag; f + g ^\sdag; g &= f ^\sdag; g + g ^\sdag; f
\\
\Leftrightarrow  \hspace{35.5pt}  (f ^\sdag - g ^\sdag); f &= (f ^\sdag - g ^\sdag) ; g
\\
\Leftrightarrow \hspace{10pt} (f - g) ^\sdag; (f - g) &= 0
\end{align*}
In a general \dag\-category with \dag\-limits, however, the exchange lemma is more general, since although hom-set addition will be cancellable by Lemma~\ref{additivelemma}, it will not necessarily be invertible.

It seems likely that without the \dag\-functor there would be no analogue to the results in this section. For this reason, we argue that the \dag\-functor is an important mathematical structure which deserves to be studied in its own right.

\subsection*{More general \dag\-limits}
\label{moregeneral}
The definition of \dag\-limits can be substantially generalized, allowing us to compute \dag\-limits of (almost) arbitrary diagrams rather than just those in the shape of finite forest-shaped multigraphs. In the case that our \dag\-category is \emph{unitary}, meaning that any pair of isomorphic objects have a unitary isomorphism going between them, this more general type of \dag\-limit can always be constructed from the simpler type, and in fact merely having zero objects, finite \dag\-products and finite \dag\-equalizers gives enough power to construct them. The rest of the paper does not depend on this subsection, so it can be safely skipped.

To describe this bigger class of \dag\-limits, we begin by considering arbitrary finite diagrams.
These are finite sets of systems and processes,  closed under composition, such that for every process its initial and final systems are included, and for every system its identity process included. Here is a drawing of a simple diagram, where for clarity we leave out the identity processes:
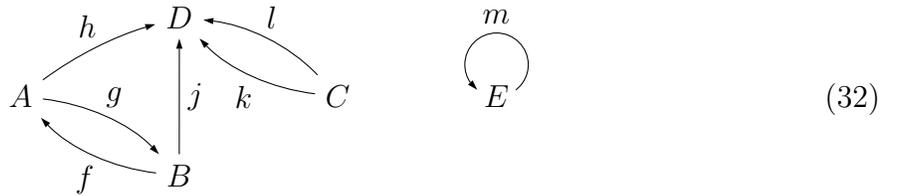
\begin{equation}
\begin{diagram}[height=30pt,width=60pt]
& \rnode 4 {D} &&
\\
\rnode 1 {A} && \rnode 3 {C} & \rnode 5 E
\\
& \rnode 2 {B} & &
\ncarc[arcangle=10]  {->} 14 \Aput h
\ncarc[arcangle=0]   {->} 24 \Bput j
\ncarc[arcangle=-20] {->} 34 \Bput l
\ncarc[arcangle=20]  {->} 34 \Aput k
\ncarc[arcangle=20]  {->} 12 \Aput g
\ncarc[arcangle=20]  {->} 21 \Aput f
\nccircle[nodesep=3pt] {->} 5 {12pt} \Bput m
\end{diagram}
\end{equation}

\jnoindent
Suppose that  these processes compose in the following way:
\renewcommand\tempgap{\hspace{30pt}}
\vspace{-5pt}
\begin{equation}
g;j  = h
\tempgap
f;h =j
\tempgap
g;f = \id_{A}
\tempgap
f;g = \id_{B}
\tempgap
m;m = m
\vspace{-5pt}
\end{equation}
Then our processes are closed under composition, and the diagram is well-defined. Note that we allow cycles in these more general diagrams, as long as we make sure to retain closure under composition.

We now choose a privileged subset $\Omega$ of the systems in the diagram, called the \emph{supporting subset}, and we refer to its elements as the \emph{supporting objects}. The only constraint we impose on $\Omega$ is that, by starting at systems in $\Omega$ and following processes in the diagram, we must be able to reach every system. So $\{A,C,E\}$ would be an allowed choice for $\Omega$, as we can get to $B$ by following $g:A \to B$, and to $D$ by following $h:A \to D$ (or alternatively $k$ or~$l$.) An illegal choice for $\Omega$ would be $\{C,D,E\}$, as neither $A$ nor $B$ can be reached starting from those objects. It is always valid to take $\Omega$ to contain all the objects in the diagram. However, it is vital that we have the freedom to take $\Omega$ as any supporting subset, not only the maximal one: otherwise we would not be able to construct \dag\-equalizers, which we rely on for many of our results. \label{nonmaxomega}

Given a particular diagram $F:\cat J \to \cat C$, and a valid choice of supporting subset $\Omega$ of the objects of \cat J, a \dag\-limit for this diagram is a limit system $L$ in the usual sense, equipped with limit maps $l_S: L \to F(S)$  satisfying the following normalization condition:
\begin{equation}
\sum _{S \in \Omega} l_S ^\pdag ; l_S ^\sdag = \id _{L}.
\end{equation}
This is very similar to the previous definition of \dag\-limits, but our normalization condition does not involve the limit maps to the leaf objects (as our diagrams will not in general be forest-shaped), but rather to the objects in the supporting subset. We draw an example of this for the example diagram given earlier, with the supporting subset chosen to be~$\Omega = \{A,C,D,E\}$:
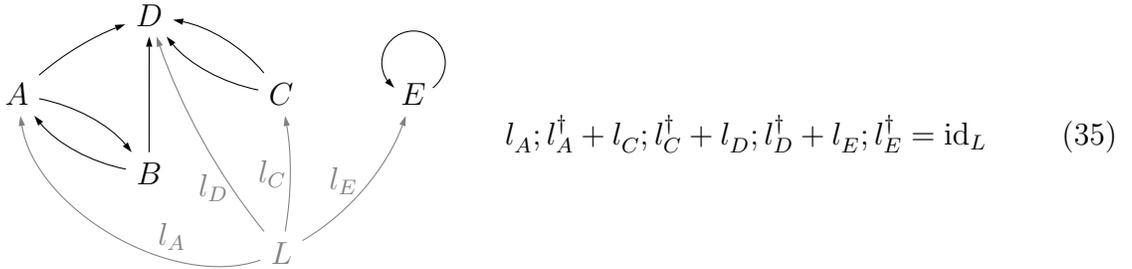
\begin{equation}
\begin{diagram}[height=30pt,width=50pt]
& \rnode 4 {D} &&
\\
\rnode 1 {A} && \rnode 3 {C} & \rnode 5 E
\\
& \rnode 2 {B} & &
\\
&& {\gray \rnode{L}{L}}
\ncarc[arcangle=10]  {->} 14
\ncarc[arcangle=0]   {->} 24
\ncarc[arcangle=-20] {->} 34
\ncarc[arcangle=20]  {->} 34
\ncarc[arcangle=20]  {->} 12
\ncarc[arcangle=20]  {->} 21
\nccircle[nodesep=3pt] {->} 5 {12pt}
\psset{linecolor=gray}
{\gray
\ncarc[arcangle=50]{->}L1 \nbput[npos=0.3]{\!\!l_A}
\ncarc[arcangle=-10]{->}L3 \Aput {l_C\!}
\ncarc[arcangle=10]{->}L4 \naput[npos=0.32]{l_D\!\!\!}
\ncarc[arcangle=-20]{->}L5 \naput[npos=0.4]{l_E\!\!}
}
\end{diagram}
\hspace{30pt}
l_A ^\pdag; l_A ^\sdag + l_C ^\pdag; l_C ^\sdag + l_D ^\pdag; l_D ^\sdag + l_E ^\pdag; l_E ^\sdag = \id _L
\end{equation}
Any \dag\-limit obtained from a forest-shaped multigraph, as described in previous sections, is clearly also a \dag\-limit in this more general sense, where the supporting subset $\Omega$ is taken to be the set of leaves of the diagram. We also mention that it is straightforward to prove an extension of Lemma~\ref{uniqueness} showing that these more general types of \dag\-limit are unique up to unique unitary isomorphism. This more general type of \dag\-limit can be computed for any diagram that admits a finite set of supporting objects.

One approach to the standard theory of categorical limits \cite{ml97-cwm} states that a category has limits exactly when the diagonal functor $\Delta: \cat C \to \cat C ^{\cat J}$ has a left adjoint. It would be desirable to find a generalization of this approach that works for the case of \dag\-limits, perhaps by replacing categories by \dag\-categories throughout. However, the author has been unable to develop a theory along these lines. One problem that is encountered is that, in the theory of \dag\-limits presented here, \emph{non}-\dag\-categories are still important --- for example, as the diagram category for a \dag\-equalizer. \label{generalizedlimitdiscussion}

\subsubsection*{Importance of the choice of supporting subset}

The maps from the limit object for a general  \dag\-limit depend significantly on the choice of supporting subset $\Omega \subseteq \Ob ( \cat J)$. As an example, consider the following simple diagram in \cat{Hilb}, the category of Hilbert spaces:
\begin{equation}
\begin{diagram}[midshaft,width=60pt,height=50pt]
& \rnode 2 {\mathbb{C}} &
\\
\rnode 1 {\mathbb{C}^2} && \rnode 3 {\mathbb{C}^2}
\ncarc[arcangle=15]  {->} 13 \Aput {\jbigl\begin{smallmatrix} 2&0 \\ 0&1 \end{smallmatrix} \jbigr}
\ncarc[arcangle=15]  {->} 31 \Aput {\left( \begin{smallmatrix} {\scriptscriptstyle \frac{1}{2}} & 0 \\ 0&1 \end{smallmatrix} \right) }
\ncarc[arcangle=-15] {->} 21 \Bput { \left( \begin{smallmatrix} 1\\0 \end{smallmatrix} \right)}
\ncarc[arcangle=15] {->} 23 \Aput { \left( \begin{smallmatrix} 2\\0 \end{smallmatrix} \right)}
\end{diagram}
\vspace{25pt}
\end{equation}
Each of the objects has a canonical basis, and we represent the morphisms of the diagram as matrices with respect to those bases. The limit object for this diagram can be taken to be $\mathbb{C}$, regardless of the choice of supporting subset. If we take the supporting subset to only contain the object $\mathbb{C}$ in the middle of the diagram, then the \dag\-limit morphism is the linear map $1 : \mathbb{C} \to \mathbb{C}$, which clearly satisfies the normalization condition. Instead, suppose we take the supporting subset to contain all the objects of the diagram; then the limit maps are, in order of objects from left to right, $\jbigl \begin{smallmatrix}  1/ \scriptscriptstyle\sqrt 6 \\ 0 \end{smallmatrix} \jbigr : \mathbb{C} \to \mathbb{C}^2$, $1/\sqrt 6 : \mathbb{C} \to \mathbb{C}$ and $\jbigl \begin{smallmatrix}  2/{\scriptscriptstyle \sqrt 6}\\ 0 \end{smallmatrix}  \jbigr : \mathbb{C} \to \mathbb{C}^2$. It is easy to check that these also satisfy the correct normalization condition. The power of the \dag\-limit construction is that these are essentially unique, up to unique unitary isomorphism.

For any object $J \in \Ob(\cat J)$, we can associate a canonical self-adjoint morphism $l_J ^\sdag ; l_J ^\pdag : F(J) \to F(J)$. This is \textit{uniquely defined} for a given supporting subset, a property that follows straightforwardly from the fact that the \dag\-limit is unique up to unique unitary isomorphism. Note that the object $J$ does not itself have to be in~$\Omega$. For the example just described, for the case that every object is in the supporting subset, these self-adjoint morphisms are, from left to right, $\jbigl \begin{smallmatrix} {\scriptscriptstyle\frac{1}{6}}& 0 \\ 0&0 \end{smallmatrix} \jbigr$, $\frac 1 6$ and $\jbigl \begin{smallmatrix} {\scriptscriptstyle\frac{2}{3}}& 0 \\ 0&0 \end{smallmatrix} \jbigr$.

{\sloppypar
Now suppose that our \dag\-category has a well-defined notion of \textit{trace} for endomorphisms, valued in some semiring, such that $\Tr(f;g)=\Tr(g;f)$ for all oppositely-directed $f$ and $g$, and \mbox{$\Tr(h+j)=\Tr(h)+\Tr(j)$} for all $h$ and $j$ which are both endomorphisms of the same object. Restricting to objects in the supporting subset and summing over these traces, we see that
\newcommand\jj{3.3pt}
\begin{align}
\sum _{S \in \Omega} \Tr\jbigl l_S ^\sdag; l_S ^\pdag \jbigr &= \sum _{S \in \Omega} \Tr(l_S ^\pdag; l_S ^\sdag) = \Tr
\shiftbrackets{3.3pt}{
$\!\displaystyle \sum _{S \in \Omega} l_S ^\pdag; l_S ^\sdag$}
= \Tr(\id_L).
\end{align}
In many contexts the scalar $\Tr(\id_L)$ represents the \emph{size} of the object $L$, and so it is apparent that each scalar $\Tr\jbigl l_S ^\sdag; l_S ^\pdag \jbigr$ --- which in many categories will be `positive' in a suitable sense~--- indicates `how much' of $L$ arises from the object $S$. Note that although $\Tr(\id_L)$ will, in many commonly-encountered categories, necessarily be an `integer', there is no such restriction on the values $\Tr\jbigl l_S ^\sdag; l_S ^\pdag \jbigr$. Also, since every diagram has a canonical choice of supporting subset given by all the objects, this gives rise to a canonical weighting, or `measure', on the objects of the diagram. For the example described above, in order of objects from left to right, these weightings are $\frac 1 6$, $\frac 1 6$ and $\frac 2 3$, which sum to $\Tr(\id_{\mathbb{C}}) = 1$ as required.

\subsubsection*{An existence theorem for \dag\-limits}

We now examine the possibility of constructing arbitrary \dag\-limits from special ones, the \dag\-equalizers and \dag\-biproducts. We will find that this is possible as long as our category is \emph{unitary}, meaning that every pair of isomorphic objects has a unitary isomorphism going between them. This can be seen as an extension of the conventional existence theorem for limits, although the proof does not transfer straightforwardly since \dag\-limits are significantly different from ordinary limits.

We begin by examining how to obtain arbitrary finite \dag\-equalizers from simpler types of \dag\-limit.
\begin{lemma}
\label{fde}
If a \dag\-category has binary \dag\-equalizers and binary \dag\-biproducts, then it has all finite \dag\-equalizers.
\end{lemma}
\jbeginproof
Let $f_i:A \to B$ be a set of parallel arrows indexed by $i \in I$, a finite set. Then we can construct the $I$-fold \dag\-biproduct $B ^{\oplus I}$, and define a column vector $F:A \to {B ^{\oplus I}}$ as the unique morphism with the property that $F; p_i = f_i$, where
$p_i : {B ^{\oplus I}} \to B$ is the projection onto the $i$th factor. Let $\Delta: B \to B^{\oplus I}$ be the diagonal map, and construct the following \dag\-equalizer:
\begin{equation}
\begin{diagram}[midshaft]
E & \rTo^e & A & \pile{\rTo ^F \\ \rTo _{f_1; \Delta}} & B ^{\oplus I}
\end{diagram}
\end{equation}
Postcomposing with $p_i: B ^{\oplus I} \to B$ we obtain $e; F; p_i = e; f_1 ; \Delta; p_i$, which simplifies to $e; f_i = e; f_1$. It follows that $e;f_i = e;f_j$ for all $i,j \in I$, and so $e:E \to A$ is a cone for the morphisms $f_i : A \to B$. Now let $x:X \to A$ be any map such that $x;f_i = x;f_j$ for all $i,j \in I$. Then $x$ is also a cone for the morphisms $F$ and $f_1; \Delta$, and so factorizes uniquely through $e:E \to A$. It follows that the morphism $e$ is the \dag\-equalizer of the morphisms $f_i : A \to B$.
\end{proof}

We will also require the two following technical lemma, which says that we can take the `square root' of any `natural number'.
\begin{lemma}
\label{rootn}
In a unitary \dag\-category with binary \dag\-equalizers and binary \dag\-biproducts, for each object $A$ and each natural number $n$, there is an isomorphism $r :A \to A$ with $r; r ^\sdag = n \cdot \id _A$.
\end{lemma}
\jbeginproof
Write $p_i: A ^{\oplus n} \to A$ for the projection of the \dag\-biproduct onto its $i$th factor, and consider all these maps together as forming an equalizer diagram:
\begin{equation}
\begin{diagram}[midshaft,nohug]
\rnode{E}{E} &&&
\\
 &  \rnode{N}{A ^{\oplus n}} & \pile{\scriptsize \rTo ^{p_1} \\ \cdots\hspace{-5pt} \\ \scriptsize\rTo _ {p_n}} & A
\\
\rnode{A}{A} &&&
\nccurve[angleA=90,angleB=-90]{->}AE\Bput {m}
\nccurve[angleA=-180,angleB=180,ncurv=1.0,linestyle=dashed]{->}EA\Bput{u}
\nccurve[angleA=0,angleB=90]{->}EN\Aput{e}
\nccurve[angleA=0,angleB=-90]{->}AN\Bput{\Delta}
\end{diagram}
\end{equation}
The $n$-fold diagonal map $\Delta : A \to A ^{\oplus n}$ is an equalizer for these maps, since given any $x:X \to A ^{\oplus n}$ with $x;p_i = x; p_j$ for all valid $i$ and $j$, $x$ factors uniquely through $\Delta$ as $x;p_1; \Delta$. By Lemma~\ref{fde} we can construct the \dag\-equalizer of the maps $p_i$, which we denote by $e : E \to A ^{\oplus n}$. Since $e$ and $\Delta$ are both equalizers, there is a unique isomorphism $m:A \to E$ with $m; e = \Delta$; and since $A$ and $E$ are isomorphic, by unitarity of the \dag\-category, there exists some unitary morphism $u:E \to A$. Defining an endomorphism $r:=m;u : A \to A$, we see that
\begin{equation}
r; r ^\sdag = m; u; u ^\sdag; m ^\sdag = m; m ^\sdag = m; e; e ^\sdag; m ^\sdag = \Delta; \nabla = n \cdot \id_A,
\end{equation}
where in the fourth expression we have inserted the identity in the form $\id _E = e;e ^\sdag$.
Since both $m$ and $u$ are isomorphisms it follows that $r:A \to A$ is also an isomorphism.
\end{proof}

We now describe a new fundamental construction called the \emph{\dag\-intersection}. In a \dag\-category, given a finite family of isometries \mbox{$x_i:X_i \to A$}, their \emph{\dag\-intersection} is defined to be a pullback $(P, \pi_i)$ such that each of the maps $\pi_i: P \to X_i$ is an isometry. The notion of \dag\-intersection is a geometrical one: given a family of isometries representing subobjects of a given object, the \dag\-intersection is an isometry representing the intersection of all these subobjects. Of course, this intersection could be zero. We note that the \dag\-intersection of a family of isometries is \textit{not} given by their \dag\-pullback, apart from the trivial case where we are taking the \dag\-intersection of a single isometry.

We now give an existence theorem for \dag\-intersections.

\begin{lemma}
\label{dagintersections}
If a unitary \dag\-category has all binary \dag\-equalizers and binary \dag\-biproducts then it has all finite \dag\-intersections.
\end{lemma}
\jbeginproof
Let $x_i : X_i \into A$ be our family of isometries in a unitary \dag\-category \cat{C}, indexed by a finite set $J$. We construct the \dag\-biproduct ${\xbigoplus} _{i \in J} X_i$, with canonical projections $p_i:{\xbigoplus} _{i \in J} X_i \to X_i$. Considering our family of isometries as a diagram in \cat{C}, we can construct its \dag\-pullback by forming the \dag\-equalizer $e:E \to {\xbigoplus} _{i \in J} X_i$ of the morphisms $p_i;x_i: {\xbigoplus} _{i \in J} X_i \to A$, making use of Lemma~\ref{fde}. The cone maps of the \dag\-limit are then given by $e;p_i:E \to X_i$. It is straightforward to check that they form a limit, and the normalization condition is satisfied since $\sum_{i \in J} e; p_i^\pdag ; p_i ^\sdag; e ^\sdag = e;( \sum _{i \in J} p_i ^\pdag ; p_i ^\sdag ); e ^\sdag = e; e ^\sdag = \id _E$.

Any of the composites $e;p_i;x_i:E \to A$, all of which are equal, intuitively represents the intersection of the isometries $x_i:X_i \to A$. However, these composites are not isometries in general; we must add a normalization factor. We construct the \dag\-intersection of the morphisms $x_i$ as $s := r_{E, |J|}; e; p_i; x_i : E \to A$, for any choice of $i \in J$, where $r _{E,|J|} : E \to E$ is an isomorphism satisfying $r _{E,|J|} ^{}; r _{E,|J|} ^{\scriptscriptstyle\dagger} = |J| \cdot \id_E$ as described in Lemma~\ref{rootn}, and $|J|$ is the number of elements of $J$. Our morphism $s$ does indeed factor through the projections of a pullback in the necessary way, since we have already shown that the morphisms $e; p_i$ form the projections of a \dag\-pullback, and since limits are preserved by isomorphisms, so do the morphisms $r_{E,|J|};e;p_i$. To show that $s$ is an isometry is to show that $s;s ^\sdag = \id _E$, and by Lemma~\ref{addlemma}, it suffices to show that $|J| \cdot (s;s ^\sdag) = |J| \cdot \id _E$: 
\begin{align}
\nonumber
|J| \cdot (s; s ^\sdag)
&= \sum _{i \in J} s; s ^\sdag
\\
\nonumber
&= \sum_{i \in J} \left((r_{E,|J|}; e; p_i; x_i); (r_{E,|J|}; e; p_i; x_i) ^\sdag \right)
\\
\nonumber
&= \sum _{i \in J} \left( r_{E,|J|} ^\pdag; e; p_i ^\pdag; x_i ^\pdag; x_i ^\sdag; p_i ^\sdag; e ^\sdag; r_{E,|J|} ^\sdag \right)
\\[-4pt]
&= r_{E,|J|} ^\pdag; e; \shiftbrackets{3.4pt}{$\displaystyle \sum_{i \in J} p _i^\pdag  ; p _i ^\sdag$} ; e ^\sdag; r_{E,|J|} ^\sdag
\nonumber
\\
&= r_{E,|J|} ^\pdag; e; e ^\sdag; r_{E,|J|} ^\sdag
\nonumber
\\
&= r_{E,|J|} ^\pdag; r_{E,|J|} ^\sdag = |J| \cdot \id _E.
\end{align}
This completes the proof.
\end{proof}

Finally, we weave these lemmas together to obtain  an existence theorem for \dag\-limits.

\begin{theorem}[Existence theorem for \dag\-limits]
\label{existence}
A unitary \dag\-category has all finite \dag\-limits iff it has a zero object, binary \dag\-equalizers and binary \dag\-biproducts.
\end{theorem}

\jbeginproof
If a \dag\-category has all finite \dag\-limits then it has these three constructions; the zero object is the \dag\-limit of the empty diagram, binary \dag\-equalizers are manifestly \dag\-limits, and binary \dag\-biproducts are \dag\-limits by Lemma~\ref{dagbiproducts}.

Conversely, consider a unitary \dag\-category \cat{C} with a zero object, binary \dag\-equalizers and binary \dag\-biproducts. By Lemma~\ref{fde} such a category actually has all finite \dag\-equalizers, and it is straightforward to obtain all finite \dag\-biproducts from binary \dag\-biproducts. Since finite biproducts exist the category is enriched in commutative monoids, and so the notion of a \dag\-limit is well-defined. Consider a diagram $F:\cat{J} \to \cat{C}$, with a chosen supporting subset $\Omega \subseteq \Ob(\cat J)$. We will show that this has a \dag\-limit.

If $\Omega$ is empty then \cat{J} must also be empty, and the \dag\-limit of $F$ is given by the zero object in \cat{C}. Otherwise, form the \dag\-biproduct in \cat{C} of the images $F(S)$ of the objects in the supporting subset, for all $S \in \Omega$. We denote this \dag\-biproduct by $\xbigoplus _{F(\Omega)}$, and write the projections onto the factors as $p_S: \xbigoplus _{F(\Omega)} \to F(S)$ for all $S \in \Omega$.

For each $T \in \Ob(\cat J)$, denote by $A_T$ the set of arrows in \cat{J} which go from an object in $\Omega$ to $T$, and for each arrow $f \in A_T$ denote its domain supporting object by $\sigma(f) \in \Omega$, so we have $f:\sigma(f) \to T$. For each $f \in A_T$, we can construct a morphism $[f]:\xbigoplus _{F(\Omega)} \to F(T)$ as the following composite:
{\newcommand\ts{\hspace{4pt}}
\begin{equation}
\textstyle [f]:=\xbigoplus _{F(\Omega)}
\rTo ^{\textstyle \ts p _{\sigma(f) \ts}}
F(\sigma(f))
\rTo ^{\textstyle \ts F(f)\ts}
F(T)
\end{equation}}Let $e_T : E_T \to \xbigoplus _{F(S)}$ be the \dag\-equalizer in \cat{C} of the arrows $[f]$ for all $f \in A_T$.

Our candidate for the \dag\-limit is the \dag\-intersection of the isometries $e_T$, over all objects $T \in \Ob(\cat J)$. We denote this \dag\-intersection by \mbox{$\pi_T; e_T:P \to {\xbigoplus} _{F(\Omega)}$}, which has the same value  for any $T \in \Omega$; the morphisms $\pi_T:P \to E_T$ are a family of isometric pullback projections, which are guaranteed to exist by Lemma~\ref{dagintersections}. The \dag\-limit maps to the objects in the supporting subset are \mbox{$l_S:=\pi_T; e_T;p _S:P \to F(S)$} for any $T \in \Ob(\cat J)$, and  for all $S \in \Omega$.

We must show that these maps form a universal, normalized cone for the diagram. First, we show that the maps $l_S: P \to F(s)$ satisfy the normalization condition~(\ref{norm}):
\begin{align}
\nonumber \sum_{S \in \Omega} l_{S} ^\pdag; l _{S} ^\sdag
&= \sum_{S \in \Omega} \pi_T; e_T; p_S; (\pi_T; e_T; p_S) ^\sdag
\\
&= \sum_{S \in \Omega} \pi_T ^\pdag; e_T ^\pdag; p_S ^\pdag; p_S ^\sdag; e_T ^\sdag; \pi_T ^\sdag
\nonumber
\\[-4pt]
&= \pi_T ^\pdag; e_T ^\pdag; \shiftbrackets{3.4pt}{$\displaystyle \sum_{S \in \Omega} p _{S}^\pdag  ; p _{S} ^\sdag$} ; e_T ^\sdag ; \pi_T ^\sdag
\nonumber
\\
&= \pi_T ^\pdag; e_T ^\pdag; e_T ^\sdag; \pi_T ^\sdag = \pi_T ^\pdag; \pi_T ^\sdag = \id _E.
\end{align}

\jnoindent
To establish that the morphisms $l_S$   define a cone, we must show that the equation \mbox{$l_ {\sigma(f)};F(f) = l_{\sigma(g)}; F(g)$} is satisfied for all $T \in \Ob(\cat J)$ and all $f,g \in A_T$. By the definition of the cone maps \mbox{$l_{\sigma(f)} ; F(f) = \pi_T; e_T; [f]$}, and since $e_T;[f] = e_T;[g]$ we see that the cone property holds. To establish the universal property, consider a cone of morphisms $x_S:X \to F(S)$ for all $S \in \Omega$; the cone property is that for all $T \in \Ob(\cat J)$ and all $f,g \in A_T$, we have $x _{\sigma(f)} ; F(f) = x _{\sigma (g)} ; F(g)$. Let $\widetilde x: X \to {\xbigoplus} _{F(S)}$ be the unique morphism such that $\widetilde x; p _S = x_S$ for all $S \in \Omega$. Then by the cone property, for all $T \in \Ob (\cat J)$ and all $f,g \in A_T$ we have $\widetilde x; [f] = \widetilde x ; [g]$, and so for all $T \in \Ob(\cat J)$ there is a unique morphism $\chi_T : X \to E_T$ with $\widetilde x = \chi_T; e_T$. Since $(P,\pi_T)$ form a pullback of the morphisms $e_T$, there must in turn be a unique morphism $\widetilde \chi : X \to P$ such that $\widetilde \chi; \pi_T = \chi_T$. Since each $e_T$ has a retraction, $\widetilde \chi$ is also the unique morphism with the property that $\widetilde \chi; \pi_T; e_T = \chi_T; e_T= \widetilde x$. It follows that $\widetilde \chi$ is the unique morphism with $\widetilde \chi; \pi_T; e_T; p_S = \widetilde x; p_S$ for all $S \in \Omega$, and so it is also the unique morphism with $\widetilde \chi; l_S = x_S$. So $(P; l_S, S \in \Omega)$  indeed gives a \dag\-limit for the diagram $F: \cat{J} \to \cat{C}$, with $\Omega$ the supporting subset.
\end{proof}

\section{Embedding the scalars into a field}
\label{mainsection}

Our main theorem of this section is stated most naturally in a \emph{monoidal \dag\-category}. Conventionally, this means a monoidal category which is also a \dag\-category, such that the unit and associator natural isomorphisms are unitary. While this gives the category nicer properties as a whole, we will not need to use them. So, for our purposes, a monoidal \dag\-category can be simply taken to mean a monoidal category which is also a \dag\-category.

In any monoidal category, we define the \emph{scalars} to be the hom-set $\Hom(I,I)$. This will have a certain amount of extra structure, depending on the properties of the ambient category. At the very least, as is well-known, it is a commutative monoid, where monoid multiplication is given by morphism composition.

Our main result concerns the scalars in a monoidal \dag\-category with all finite \dag\-limits, which have the structure of a semiring with involution. We will prove the following theorem:

\begin{theorem}
\label{fieldtheorem}
In a nontrivial monoidal \dag\-category with simple tensor unit, and with all finite \dag\-limits, the involutive semiring of scalars has an involution-preserving embedding into an involutive field with characteristic 0 and orderable fixed field.
\end{theorem}

\jnoindent
The proof of this theorem will be given piece-by-piece throughout this section. Just to be clear, by `field' we mean a classical algebraic field: a commutative ring with multiplicative inverses for every nonzero element. By `characteristic 0' we mean that no finite sum of the form $1+1 + \cdots + 1$ gives zero. By `involutive semiring' and `involutive field' we mean a structure equipped with an order-2 automorphism that respects addition and multiplication, and by `fixed field' we mean the subfield on which the automorphism acts trivially. By `simple tensor unit', we mean that every monic map into the tensor unit is either zero or an isomorphism; in other words, it has no proper subobjects.

The connection between this theorem and the complex numbers is given by the following well-known characterization of the subfields of the complex numbers.\footnote{This theorem is often considered surprising, given that there seem to be `obvious' counterexamples: for example, a field of \emph{rational functions}, which has elements given by equivalence classes of ratios of complex polynomials $P(x)/Q(x)$, where $Q(x)$ is not the zero polynomial, and where \mbox{$P(x)/Q(x) \sim P'(x) / Q'(x)$ when $P(x) Q'(x) = P'(x) Q(x)$}. An embedding of such a field into the complex numbers is difficult to visualize, since it will be highly noncontinuous with respect to the natural topologies involved. To prove the theorem, take any field of characteristic~0 and at most continuum cardinality, and add to it a continuum of transcendentals, obtaining a field of precisely continuum cardinality. Then take the algebraic completion. The result is isomorphic to the complex numbers, since it is an algebraically-closed field of characteristic~0 and continuum cardinality.}
\begin{theorem}
The subfields of the complex numbers are precisely the fields of characteristic~0 which are at most of continuum cardinality.
\end{theorem}

\jnoindent
It follows immediately that, if we have a monoidal \dag\-category satisfying the conditions of Theorem~\ref{fieldtheorem} for which the scalars are at most continuum cardinality, they must embed as a semiring into the complex numbers. However, we cannot guarantee that there will be an \emph{involution-preserving} embedding into the complex numbers, which translates the action of the \dag\-functor on the scalars into complex conjugation on the complex numbers. We deal with this in the next section.

In addition to the \dag\-limits which we studied in the previous section, Theorem~\ref{fieldtheorem} requires two extra conditions: nontriviality, and that the monoidal unit object is simple. Both are natural, in the sense that they prevent the theorem from being `obviously' false. A field is required to have $0 \neq 1$, and this translates to the condition that our category is nontrivial. Also, if we had a monoidal \dag\-category satisfying the conditions of the theorem, we could take the cartesian product of this category with itself; this has an obvious monoidal structure for which the monoidal unit \textit{does} have proper \dag\-subobjects, the scalars being pairs of scalars in the original category. Such a semiring can never embed into a field, since it contains \emph{zero divisors}, nonzero elements $a$ and $b$ which satisfy $ab=0$. Requiring the monoidal unit to lack proper subobjects blocks this obvious source of counterexamples.

\subsection*{The scalars as a semiring}

We begin by showing that the scalars in a monoidal category form a commutative monoid. We establish this with the classic argument due to Kelly and Laplaza \cite{kl80-cccc}, related to the Eckmann-Hilton argument\label{eharg}. We note that this commutativity property is the only reason that we prove Theorem~\ref{fieldtheorem} for the scalars in a monoidal category; it would hold for any commutative endomorphism monoid on an object without proper \dag\-subobjects.
\begin{lemma}
\label{eh}
In a monoidal category, the scalars are commutative.
\end{lemma}

\jbeginproof
We present the standard commutative diagram in the form of a cube, which holds for any two scalars $a,b:I \to I$. The coherence equation $\rho _I = \lambda _I$ is essential.
\begin{equation}
\begin{diagram}[midshaft,nohug,width=35pt,height=15pt]
I & \rTo^a &&& I &&
\\
& \rdTo _b & {\vrule height 25pt width 0pt}  && \vLine & \rdTo^b &
\\
&& I &\rTo^{a} & \HonV && I
\\
\dTo <{\lambda _I ^{}} &&& \hspace{80pt} & \dTo <{\lambda_I ^{}} >{\rho _I ^{}} &&
\\
&&& {\vrule height 25pt width 0pt}&&&
\\
&&\uTo <{\lambda _I ^{-1}} > {\rho ^{-1}_I} &&&& \uTo >{\rho_I ^{-1}}
\\
I \otimes I & \hLine  & \VonH & \rTo _{\hspace{-20pt}a \otimes \id _I\hspace{-20pt}}& I \otimes I && 
\\
& \rdTo_{\id_I \otimes b} && {\vrule height 20pt width 0pt} && \rdTo ^{\id_I \otimes b} &
\\
& & I \otimes I & \rTo _{a \otimes \id_I} && & I \otimes I
\end{diagram}
\end{equation}
\vspace{-30pt}

\end{proof}

We next show that, if the monoidal category also has biproducts, the scalars form a commutative semiring. A semiring, sometimes called a \emph{rig}, is a structure similar to a ring but which is not required to have have additive inverses for all elements. In this paper a ring always has a multiplicative unit, and the zero element satisfies $0x=x0=0$ for all elements $x$ in the ring.

In a category with biproducts, the hom-sets have a commutative monoid structure as described by equation (\ref{cmonenrichment}). Interpreting this monoid structure as addition, and composition of scalars as multiplication, these structures combine to give the scalars in a monoidal category with biproducts the structure of a commutative semiring. To prove this, we need to show that $a (b+c) = (a  b) + (a  c)$ for all scalars $a,b,c$; this follows from naturality of the diagonal and codiagonal maps, as discussed earlier on page~\pageref{adddist}\label{adddiscussion}. We also require $0 a = a  0 = 0$ for all scalars $a$, which follows from the definition of the zero morphisms.

This commutative semiring of scalars acts in a natural way on the hom-sets of the category. For any morphism $f:A \to B$ and any scalar $a$, we define $a \cdot f: A \to B$ as follows:
\begin{equation}
\begin{diagram}[height=30pt,width=45pt]
A & \rTo ^{a \cdot f} & B
\\
\dTo < {\lambda_A} && \uTo >{\lambda _B ^{-1}}
\\
I \otimes A & \rTo_{a \otimes f} & I \otimes B
\end{diagram}
\end{equation}
In fact, this gives the each hom-set the structure of a \emph{semimodule} over the scalars, which is the natural notion of module extended from a ring to a semiring.

We now consider the extra structure given by the \dag\-functor and \dag\-biproducts. The \dag\-functor gives us an involution on the scalars, sending $a:I \to I$ to $a ^\sdag:I \to I$. This involution is order-reversing for multiplication, due to the contravariance of the \dag\-functor, and distributes over addition as explained in the discussion around equation (\ref{dagadd}). This gives the scalars the structure of an \emph{involutive semiring}. In the case that the unit isomorphisms associated to the monoidal structure are unitary, the hom-sets then become \emph{involutive semimodules} for this semiring, but we will not need this extra structure.

One aim of this research is to understand the categorical structure of the complex numbers, which is certainly an involutive semiring, so the category theory is generating the correct kind of structure. Of course, the complex numbers are far more than just a semiring, and we will now see how some of the necessary extra properties arise.

\subsection*{Embedding into a field}

To achieve our goal of embedding the scalars into a field, it is clear that additive cancellability is a necessary property. We demonstrated this for all hom-sets in \dag\-categories with finite \dag\-biproducts and finite \dag\-equalizers in Lemma~\ref{additivelemma}. Another property which is clearly necessary is cancellable multiplication.
\begin{defn}
\label{cancmult}
A commutative semiring has \emph{cancellable multiplication} when, for any three elements $a,b,c$ in the semiring, $ac=bc, \, c \neq 0 \, \Rightarrow \,a=b$.
\end{defn}

\jnoindent
We now show that the scalars have this property in any category of the type which we are considering. The condition that the monoidal unit has no proper \dag\-subobjects is clearly crucial here, but this is far from the only role played by this condition in proving the theorem.
\begin{lemma}
\label{lemmacancmult}
In a monoidal \dag\-category with simple tensor unit, a zero object and finite \dag\-equalizers, the scalars have cancellable multiplication.
\end{lemma}

\jbeginproof
Suppose that the scalars did not have cancellable multiplication. Then there would exist scalars $a,b,c$ with $c \neq0$, such that $a \neq b$ but $ac=bc$. We consider the following commuting diagram:
\begin{equation}
\begin{diagram}[nohug,midshaft,width=20pt,height=35pt]
&
&\rnode{1}{I}&
&&&
\\
\rnode{2}{E} && \rInto^{e} && \rnode{3}{I} & \hspace{35pt} \pile{ \rTo ^a \\ \rTo _b} & I
\nccurve[angleA=180,angleB=80]{->}12\Bput {\tilde c}
\nccurve[angleA=0,angleB=100]{->}13 \Aput c
\end{diagram}
\end{equation}
The \dag\-equalizer morphism $e:E \to I$ gives a \dag\-subobject of $I$. It is not zero, since $c$ factors through it and $c\neq 0$; also, since $a \neq b$, it cannot be an isomorphism. It follows that $I$ has a proper \dag\-subobject, but this contradicts our hypothesis. It follows that the scalars have cancellable multiplication.
\end{proof}

As a first step towards embedding the scalars into a field, we first embed them into a ring. Given our semiring $S$ of scalars, we can construct its \emph{difference ring} $D(S)$. Elements of $D(S)$ are equivalence classes of ordered pairs $(a,b)$ of elements of $S$, which we write using the suggestive notation $a-b$. The equivalence relation is given by
\begin{equation}
a-b \sim c-d \quad \textrm{iff} \quad a+d = c+b.
\end{equation}
It is a standard exercise to show that this is symmetric, transitive and reflexive, for which we rely on the fact that the scalars have cancellable addition. Addition and multiplication are defined on representatives of the equivalence classes in the familiar algebraic way:
\begin{align}
(a-b)+(c-d) &= (a+c)-(b+d)
\\
(a-b)(c-d) &= (ac+bd)-(ad+bc)
\end{align}
These are well-defined on equivalence classes.

We see that the scalars in our category embed into their difference semiring, under the obvious mapping $a \mapsto a-0$. For two elements to be sent to the same element of the difference ring would mean that $a-0\sim b-0$, but applying the definition of the equivalence relation then gives $a=b$, so the mapping is faithful.

As we will see, the difference ring embeds into a field if and only if it has cancellable multiplication. From Definition~\ref{cancmult}, this condition is
\begin{equation}
(a-b)(c-d) \sim (a-b)(e-f), \,\, a-b \nsim 0 \, \Rightarrow \, c-d\sim e-f
\end{equation}
for all choices of elements $a,b,c,d,e,f \in S$. Using the definition of the equivalence relation to write this directly in terms of the elements of the underlying semiring, we obtain
\begin{equation}
a(c+f)+b(d+e) = a(d+e) + b(c+f), \, a \neq b \, \Rightarrow \, c+f = d+e.
\end{equation}
Defining $A := c+f$ and $B := d+e$, this reduces to the condition
\begin{equation}
aA + bB = aB+bA, \, a \neq b \, \Rightarrow \, A = B.
\end{equation}
We now show that this holds in any category of the type we are working with. In some ways, this condition resembles that of the exchange lemma~\ref{exchangelemma}, but it is logically independent from it.

\begin{lemma}
\label{ABablemma}
In a monoidal \dag\-category with simple tensor unit and all finite \dag\-limits, any choice of scalars $A,B,a,b : I \to I$ satisfies the implication
\[
aA+bB=aB+bA, \, a \neq b \,\Rightarrow\, A=B.
\]
\end{lemma}

\jbeginproof
We have already shown that the scalars in such a category are commutative and have cancellable addition and multiplication, and we will use these properties throughout. Let $A,B,a,b$ be scalars satisfying $aA+bB = aB+bA$ and $a \neq b$. If $a=0$ then $bB=bA$, and cancelling the nonzero $b$ we obtain $B=A$; the case $b=0$ is similar. Conversely, if $A=0$ then $bB=aB$, and $B=A=0$ is the only possibility, or $B$ would cancel contradicting our assumption that $a \neq b$; the case $B=0$ is similar. In each of these cases, therefore, the implication holds.

We now consider the case in which none of the four scalars are zero. We construct the following commutative diagram where $(E,e)$ is a \dag\-equalizer for the parallel pair $(A \,\,\, B)$ and $(B \,\,\, A)$, and $(I,p)$ and $(I,q)$ are cones:
\begin{equation}
\begin{diagram}[nohug,midshaft,width=30pt,height=40pt]
&
\mathmove{5pt}{-25pt}{\tilde p}
&\rnode{1}{I}&
\mathmove{5pt}{13pt}{p = \big(
\hspace{-2pt}
{
\begin{smallmatrix} a \\[2pt] b \end{smallmatrix}
}
\hspace{-2pt} \big)
\hspace{-43pt}
}
&&&
\\
\rnode{2}{E} && \rInto^{e= \big(
\hspace{-2pt}
{
\begin{smallmatrix} e_1 \\[2pt] e_2 \end{smallmatrix}
}
\hspace{-2pt} \big)} && \rnode{3}{I \oplus I} & \hspace{35pt} \pile{ \rTo ^{(A \,\,\, B)} \\ \rTo _{(B \,\,\, A)}} & I
\\
& \mathmove{-3pt}{-25pt}{\tilde q} & \rnode{4}{I} & \mathmove{-3pt}{13pt}{q = \big(
\hspace{-2pt}
{
\begin{smallmatrix} b \\[2pt] a \end{smallmatrix}
}
\hspace{-2pt} \big)
\hspace{-43pt}
}
 &&&
\nccurve[angleA=180,angleB=80]{->}12
\nccurve[angleA=0,angleB=100]{->}13
\nccurve[angleA=180,angleB=-80]{->}42
\nccurve[angleA=0,angleB=-100]{->}43
\end{diagram}
\end{equation}
For each cone, we denote the unique factorization through the equalizer with a tilde. Using the matrix calculus and the \dag\-equalizer equation $e; e^\sdag = \id_E$ we see that $p = p; e ^\sdag; e$ and $q = q; e^\sdag; e$, and writing these out in components, we obtain the following:
\begin{align}
a &=  a; e_1^\sdag ; e ^\pdag _1 + b; e_2 ^\sdag ; e^\pdag _1\label{xeqn}
\\
\label{beqn}
b &= a; e_1 ^\sdag; e^\pdag _2 + b; e_2 ^\sdag; e^\pdag _2
\\[5pt]
\label{yeqn}
b &= b; e_1 ^\sdag ; e_1^\pdag + a; e_2 ^\sdag; e_1 ^\pdag
\\
\label{aeqn}
a &= b; e_1 ^\sdag; e_2 ^\pdag + a; e_2 ^\sdag; e_2 ^\pdag
\end{align}
The first two equations come from the components of $p$, and the second two from the components of $q$.

Multiplying equation (\ref{xeqn}) by $b$ and (\ref{yeqn}) by $a$ and equating the right-hand sides, this \mbox{gives\hspace{20pt}}
\begin{equation}
ba; e_1 ^\sdag; e^\pdag _1 + b^2; e_2 ^\sdag; e^\pdag _1
=
ab; e_1 ^\sdag; e^\pdag _1 + a^2; e_2 ^\sdag; e^\pdag _1.
\end{equation}
We apply commutativity and additive cancellability to obtain
\begin{equation}
b^2; e_2 ^\sdag; e^\pdag _1 = a^2; e_2 ^\sdag; e^\pdag _1.
\end{equation}
We note that the quantity $e_2 ^\sdag; e^\pdag _1$ is a scalar. Either it is zero, or it is nonzero and it can be cancelled to give $a^2 = b^2$. We will consider these cases separately. First we assume that $e_2 ^\sdag ; e^\pdag _1 \neq 0_{I,I}$ and $a^2=b^2$. Defining $c := a+b$, we see that
\begin{align}
ca &= a^2 + ba = b^2 + ab = cb.
\end{align}
So $ca = cb$, and if $c \neq 0$ it will cancel from both sides to give $a=b$. However, by assumption $a \neq b$, and so we must have $c=0$ and $a+b=0$. Returning to our equation $aA + bB = aB+bA$ and adding $b(A+B)$ to both sides, we obtain
\begin{align}
aA + bB + b(A+B) &= aB + bA +b(A+B)
\nonumber
\\
\Rightarrow  (a+b)A + 2bB &= (a+b)B + 2 bA
\nonumber
\\
\Rightarrow 2bB = 2bA. \hspace{19.5pt}&
\end{align}
Since $2 : I \to I$ is given by $\Delta _I; \nabla _I = \Delta _I; (\Delta _I) ^\sdag$ where $\Delta_I : I \to I \oplus I$ is the diagonal for the biproduct, by Lemma~\ref{nondegen} it must be nonzero, and so it can be cancelled from both sides. By assumption $b \neq 0$, and so it can be cancelled as well. This gives $B=A$ as required. The only unresolved case is $e_2 ^\sdag; e_1 ^\pdag = 0$.

Alternatively, we could have multiplied equation (\ref{beqn}) by $a$ and equation (\ref{aeqn}) by $b$ and equated the right-hand sides. This leads to a similar conclusion: either $e_1 ^\sdag; e_2 ^\pdag \neq 0$ and $A=B$, or $e_1 ^\sdag; e_2 ^\pdag = 0$ and the theorem is not immediately resolved. Since this line of argument is independent from the previous one, the only remaining case to consider is that $e_2 ^\sdag; e_1 ^\pdag = e_1 ^\sdag; e_2 ^\pdag = 0$.

We have not yet used the fact that the equalizer $e: E \to I \oplus I$ is a cone, which is asserted by the following equation:
\begin{equation}
e_1; A + e_2; B = e_1; B + e_2; A.
\end{equation}
Composing on the left with $e_2 ^\sdag$, we obtain
\begin{equation}
e_2^\sdag; e^\pdag _1; A + e_2 ^\sdag; e^\pdag _2; B = e_2 ^\sdag; e^\pdag _1; B + e_2 ^\sdag; e^\pdag _2; A.
\end{equation}
Applying $e_2 ^\sdag; e^\pdag _1 = 0$, this gives
\begin{equation}
e_2 ^\sdag; e^\pdag _2; B = e_2 ^\sdag; e^\pdag _2; A.
\label{YXeqn}
\end{equation}
To deal with this we need to know the value of the scalar $e_2 ^\sdag; e^\pdag _2$. We observe that $\Delta _I = \jbigl \hspace{-1pt} \begin{smallmatrix} \id_I \\ \id_I \end{smallmatrix} \hspace{-1pt} \jbigr:I \to I \oplus I$ is a cone, and so there exists some $\widetilde \Delta _I:I \to E$ satisfying $\widetilde \Delta _I; e = \Delta _I$. Using the \dag\-equalizer equation $e; e^\sdag  = \id _E$ we obtain $\widetilde \Delta _I = \Delta _I; e_{\pdag} ^\sdag= e_1 ^\sdag + e_2 ^\sdag$. Postcomposing with $e_2$ gives the equation
\begin{align}
\widetilde \Delta _I; e_2 &= 1 = e_1 ^\sdag ; e^\pdag _2 + e_2 ^\sdag ; e^\pdag _2.
\label{reqn}
\end{align}
Applying the assumption that $e_1 ^\sdag; e^\pdag _2 = 0$, this gives $e _2 ^\sdag; e^\pdag _2 = 1$. Equation (\ref{YXeqn}) then gives $B=A$ as needed, which completes the proof.
\end{proof}
\jnoindent
At the cost of a more long-winded proof we have avoided using the \dag\-functor explicitly here. We are certainly relying on it indirectly, however, as we require that addition in the semiring is cancellable; this was proved in Lemma~\ref{additivelemma}, and it does not seem that the use of the \dag\-functor in that proof can be avoided.

For any nontrivial commutative ring $R$ with cancellable multiplication, we can obtain its \emph{quotient field} $Q(R)$ into which $R$ embeds. Elements of $Q(R)$ are equivalence classes of pairs $(s,t)$ of elements of $R$ with $t \neq 0$. We write these pairs in the form $\frac{s}{t}$, to resemble a fraction. The equivalence relation is given by
\begin{align}
\frac{s}{t} \sim \frac{u}{v} \quad &\mathrm{iff} \quad sv=ut.
\end{align}
This is symmetric, transitive and reflexive, as required. We rely on the cancellable multiplication to demonstrate transitivity. Multiplication and addition are defined on representatives of the equivalence classes as if they were conventional fractions:
\begin{align}
\frac{s}{t} \cdot \frac{u}{v} &= \frac{su}{tv}
\\
\frac{s}{t} + \frac{u}{v} &= \frac{sv + ut}{tv}
\end{align}
These operations are well-defined on the equivalence classes. Furthermore, the ring $R$ embeds into $Q(R)$ under the mapping $r \mapsto \frac{r}{1}$, and this is faithful since $\frac{r}{1} \sim \frac{s}{1}\, \Rightarrow\, r=s$. It is straightforward to see that this embedding preserves multiplication and addition.

We require the commutative ring $R$ to be nontrivial, satisfying $0 \neq 1$, since a field must satisfy this by definition. This leads to the requirement that the monoidal category from which we obtain our scalars must be nontrivial, having more than one morphism. We must require this explicitly, since the one-morphism category otherwise satisfies all of our conditions: it is a monoidal \dag\-category with all finite \dag\-limits, for which the monoidal unit object has no proper \dag\-subobjects.

Altogether, for a nontrivial monoidal \dag\-category with all finite \dag\-limits, in which the monoidal unit is simple, we have shown that the commutative semiring $S$ of scalars embeds into the commutative difference ring $D(S)$; that this ring has cancellable multiplication; and that any ring $R$ with cancellable multiplication embeds into its quotient field $Q(R)$. It follows that the semiring $S$ embeds into $Q(D(S))$, and so the scalars in our monoidal category embed into a field.

\subsection*{Establishing the characteristic}

We next show that the semiring $S$ of scalars has characteristic~0. Since we have shown that this semiring embeds into the field $Q(D(S))$, it follows that this field must also have characteristic~0.

\begin{lemma}
\label{nzcharsub}
In a nontrivial monoidal \dag\-category with finite \dag\-biproducts and \dag\-equalizers, for which the monoidal unit object has no \dag\-subobjects, the scalars have characteristic~0.
\end{lemma}

\jbeginproof
Suppose that scalar addition is not of characteristic~0. Then there exists some nonzero scalar $a: I \to I$, and positive natural number $n$, such that
\begin{equation}
a + \cdots + a = 0
\end{equation}
where the sum contains $n$ copies of $a$. This sum is equal to $n \cdot a$, where $n:I \to I$ is a scalar given by $\Delta ^n_I ; \nabla ^n_I$, for $\Delta _I ^n$ the $n$-fold codiagonal of $I$ and $\nabla _I ^n$ the $n$-fold diagonal. From the \dag\-biproduct property it follows that $\nabla _I ^n = (\Delta _I ^n) ^\sdag$ by Lemma~\ref{daggerdcod}, and from the \dag\-equalizer property it follows in turn that $n = \Delta _I ^n ; ( \Delta _I ^n) ^\sdag \neq 0$ by Lemma~\ref{nondegen}. However, by Lemma~\ref{cancmult}, the product of two nonzero scalars cannot be zero. We conclude that our original assumption was wrong, and that scalar addition is of characteristic~0.
\end{proof}

\subsection*{Involution and ordering}

The action of the \dag\-functor gives the scalars the structure of an \emph{involutive} semiring, equipping it with an involution that respects semiring addition and multiplication: we have $(a+b) ^\sdag = a ^\sdag + b ^\sdag$ by Lemma~\ref{daggerdcod}, and $(ab) ^\sdag = a ^\sdag  b ^\sdag$ by functoriality. An involution is usually required to be order-reversing for multiplication, which is satisfied in a natural way since the \dag\-functor is contravariant, but we can neglect this here as the scalars are commutative.

The self-adjoint scalars  are those scalars satisfying $a = a ^\sdag$. These self-adjoint scalars are closed under multiplication and addition, and so form a subsemiring. It is easy to see that the field $Q(D(S))$ into which the scalars $S$ embed inherits the involution, and so is an involutive field. The self-adjoint elements of $Q(D(S))$ also form a field, and the self-adjoint scalars embed into this field. 

We now demonstrate that the self-adjoint scalars admit an \emph{order}. An order on a semiring is a reflexive total order on the underlying set, such that the following conditions hold:
\begin{align}
a \leq b \, &\Rightarrow \, a+c \leq b+c
\\
0 \leq a, 0 \leq b \, &\Rightarrow \, 0 \leq ab
\end{align}
We will not work directly with these conditions. Instead, we will take advantage of the fact that our scalars embed into a field, and use the following classical theorem on orders for fields (for a proof, see \cite[Theorem 3.3.3]{m02-mtai}.)
\begin{theorem}
\label{fieldordertheorem}
A field admits an order if and only if a finite sum of squares of nonzero elements is never zero.
\end{theorem}

\jnoindent
We will use this theorem to show that the self-adjoint elements of the field $Q(D(S))$  admit an order. It then follows straightforwardly that the semiring of self-adjoint elements of $S$ admits an order, through its involution-preserving embedding into $Q(D(S))$. However, we emphasize that there is no guarantee that this order will be unique, or that there will be a canonical choice of order.

We actually prove a more general theorem, on sums of \emph{squared norms} of elements of $Q(D(S))$.

\begin{defn}
For a field with involution $a \mapsto a ^\sdag$, the \emph{squared norm} of $a$ is $a a ^\sdag$.
\end{defn}

\begin{lemma}
\label{sumnonzerosquares}
Let $S$ be the semiring of scalars in a nontrivial monoidal \dag\-category with simple tensor unit, and with all finite \dag\-limits. Then a finite sum of squared norms of nonzero elements of the field $Q(D(S))$ is never zero.
\end{lemma}
\jbeginproof
We must show that, given any finite sum satisfying
\begin{equation}
\label{bigsum}
a_1 a_1 {}^\sdag + a_2 a_2 {}^\sdag + \cdots + a_N a_N {}^\sdag = 0
\end{equation}
where each $a_i$ is an element of $Q(D(S))$, each $a_i$ is actually zero. By construction, each $a_i$ is a formal quotient $b _i / c_i$ of some pair of elements $b_i, c_i$ in $D(S)$. Writing the sum in terms of these quotients, and multiplying through by each denominator, we obtain another sum in the form of (\ref{bigsum}) in which each term is a squared norm of an element of $Q(D(S))$ with trivial denominator; in other words, an element of $D(S)$. Writing these elements as formal ordered pairs $d_i - e_i$, where $d_i, e_i$ are elements of $S$, we obtain the sum
\begin{equation}
\label{biggersum}
(d_1 - e_1) (d_1 - e_1) ^\sdag + (d_2 - e_2)(d_2-e_2)^\sdag + \cdots + (d_N-e_N) (d_N- e_N) ^\sdag = 0.
\end{equation}

\jnoindent
We define the morphism $d: I \to I ^{\oplus N}$ to be the column vector with components $(d_1, d_2, \ldots, d_N)$, and the morphism $e:I \to I ^{\oplus N}$ to be the column vector with components $(e_1, e_2, \ldots, e_N)$. By matrix multiplication, we see that equation (\ref{biggersum}) is precisely equivalent to the equation
\begin{equation}
d; d^\sdag + e; e^\sdag = d; e^\sdag + e; d ^\sdag.
\end{equation}
We can now apply the exchange lemma~\ref{exchangelemma} to conclude that $d = e$, and so $e_i = d_i$ for all~$i$. It follows that each of the original $a_i = \frac{d_i - e_i}{\scriptstyle \textrm{denom}}$ was zero, and that the sum of squared norms was in fact a sum of zeros, which proves the lemma.
\end{proof}

\jnoindent
From this lemma we see that a finite sum of squares of nonzero self-adjoint elements of the field $Q(D(S))$ is nonzero. So by Theorem~\ref{fieldordertheorem} the self-adjoint elements of $Q(D(S))$ admit an ordering, and in general they will admit many different orderings. By extension, the self-adjoint elements of the scalar semiring $S$ also admit an ordering, since they embed into the self-adjoint elements of $Q(D(S))$. This concludes the proof of the main theorem.

\section{Completing the scalars}

We have shown that, in a monoidal \dag\-category with all finite \dag\-limits that satisfies the conditions of the previous section, the scalars share many properties with the complex numbers. In particular, the self-adjoint scalars will admit an order, just as the real numbers do. The order on the real numbers is a special one: in particular, it is \mbox{\emph{Dedekind-complete}}, which for a total order means that every subset with an upper bound has a least upper bound, and every subset with a lower bound has an greatest lower bound.

The real numbers are also a field, and the field structure interacts well with the Dedekind-completeness property of the underlying total order: if $X$ is a set of elements of $\mathbb{R}$ with upper bound $\bigvee (X)$, then we have $\bigvee(X+r)=\bigvee(X) + r$, where $r$ is a real number and $X+r$ denotes the set $\{x+r|x \in X\}$, and similarly $\bigwedge(X+r) = \bigwedge (X) + r$. If a totally-ordered semiring has a Dedekind-complete underlying totally-ordered set, and has an addition operation satisfying these extra compatibility conditions, then we call it a \emph{Dedekind-complete semiring}.

In this section, we will show that this Dedekind-completeness property is the extra abstract property required to characterize the complex numbers. To work towards this, we first prove a useful lemma.

\begin{lemma}
\label{semiringproperties}
Suppose a commutative semiring contains the positive rational numbers and is additively cancellable, multiplicatively cancellable, totally-ordered and Dedekind-complete. Then it has the following properties:
\jbeginenumerate

\item \emph{(Means.)}
\label{meanprop}
For any pair of elements $a<b$ we can construct their `mean' as $\frac 1 2 (a+b)$, which satisfies $a< \frac 1 2 (a+b) < b$.
\item \emph{(Partial subtraction.)} For any pair of positive elements $a$ and $b$ with $a<b$, there exists an element $c$ with $c+a=b$.
\label{parsub}

\item \emph{(No positive infinitesimals.)} For any positive element $a$, there exists a natural number $n$ such that $an > 1$.
\label{postiny}

\item \emph{(No positive infinite elements.)} For any positive element $a$, there exists a natural number $n$ such that $a<n$.
\label{posinf}

\item \emph{(Dense positive rationals.)} For any two unequal positive elements, there is a rational number between them.
\label{denseposrat}

\item \emph{(Real numbers.)} The semiring is isomorphic to either the semiring $\mathbb{R} ^{\geq0}$ of nonnegative real numbers, or the field $\mathbb{R}$ of all real numbers.
\label{realnum}
\end{enumerate}
\end{lemma}
\jbeginproof
We prove these properties sequentially, at times using lower-numbered properties to aid the proof of higher-numbered ones. Throughout, let $L$ be a commutative semiring satisfying the hypotheses of the lemma, and let $a,b \in L$ be variables valued in the semiring.

\emph{\ref{meanprop}.} (Means.) Since $a<b$ we have $a+a=2a<a+b$, and multiplying by the fraction~$\frac 1 2$, we obtain $a < \frac 1 2 (a+b)$. Similarly, we can also show that $\frac 1 2 (a+b) < b$.

\textit{\ref{parsub}.} (Partial subtraction.) For any pair of elements $a,b$ satisfying $0 < a<b$, consider the following sets:
\begin{align}
J &= \{x \in L,  x+a>b \}
\\
K &= \{ x \in L, x+a<b \}
\end{align}
The set $J$ has a lower bound $0$ and the set $K$ has an upper bound $b$, so the greatest lower bound $\bigwedge(J)$ and greatest upper bound ${\bigvee}(K)$ both exist by Dedekind-completeness. If $\bigvee(J) + a = b$ or $\bigwedge(K)+a=b$ then we have discovered $c$ and we are done, so suppose that neither hold. Suppose that $\bigwedge(J) + a < b$: then $\bigwedge (J+a) < b$ by the preservation of infima by addition, but this is not possible, since $b$ would then serve as a greater lower bound. Similarly, we can rule out $\bigvee(K) +a>b$. The only remaining situation is that in which $\bigvee(K) +a < b < \bigwedge(J)+a$, from which it follows by additive cancellability that $\bigvee(K) < \bigwedge(J)$. Construct the mean of $\bigvee(K)$ and $\bigwedge(J)$ as $m:= \frac 1 2 ({\bigvee}(K) + \bigwedge(J))$; then by property~\ref{meanprop},
\begin{equation}
\label{kmjineq}
\textstyle
\bigvee (K) < m < \bigwedge (J).
\end{equation}
Consider the value of $m+a$. Suppose that $m+a < b$; then $m \in K$ and       so $m \leq \bigvee(K)$, but this contradicts equation~\eqref{kmjineq}. Similarly, suppose that $m+a>b$; then $m \in J$ and         so $m \geq \bigwedge(K)$, and this again leads to a contradiction. The only remaining possibility is that $m+a=b$, and so we are done.

\textit{\ref{postiny}.} (No infinitesimals.) Consider the set
\begin{equation}
I = \{x \in L, x>0, \forall n \in \mathbb{N} \, \,nx <1 \},
\end{equation}
the elements of which we call the infinitesimals. Suppose the set $I$ is not empty; since the element $1$ serves as an upper bound, the supremum $\bigvee(I)$ must therefore exist, and will satisfy ${\bigvee}(I)>0$ since it is certainly greater than each positive infinitesimal. Suppose $\bigvee(I)$ is not itself an infinitesimal; then there exists some $m \in \mathbb{N}$ with $m \bigvee(I) > 1$, and multiplying by the rational number $\frac 1 m$ it follows that $\bigvee(I) > \frac 1 m$. But then $\frac 1 m $ serves as a lower upper bound to the infinitesimals than $\bigvee(I)$; this gives a contradiction, and so $\bigvee(I)$ must be an infinitesimal. Since $\bigvee(I)>0$ it follows that $2 \bigvee(I) > \bigvee(I)$; the quantity $2 \bigvee(I)$ is therefore not an infinitesimal, and there must exist some $p \in \mathbb{N}$ with $2 p \bigvee(I) > 1$. But since $2p$ is a natural number, $\bigvee(I)$ is not infinitesimal, and so we have a contradiction. It follows that the set $I$\ is empty.

\emph{\ref{posinf}.} (No positive infinite elements.) This property is proved in a similar way to property~\ref{postiny}. Define the set
\begin{equation}
H = \{ x \in L, \forall n \in \mathbb{N} \,\, x>n \},
\end{equation}
containing the infinite elements, and assume that it is not empty. Clearly this set has a positive lower bound given by any natural number, so by Dedekind-completeness it must have a positive greatest upper bound $\bigwedge(H)$. Since $\frac 1 2 {\bigwedge}(H) < \bigwedge(H)$ it follows that $\frac 1 2 {\bigwedge} (H)$ is not an infinite element, and so there exists some $n \in \mathbb{N}$ with $\frac 1 2 \bigwedge(H) < n$; from this we see that ${\bigwedge}(H) < 2n$, and so $\bigwedge(H)$ itself is not an infinite element. But then $2n$ is a greater lower bound for the elements of $H$, which contradicts the definition of $\bigwedge(H)$. The only remaining possibility is that the set\ $H$ is empty.

\emph{\ref{denseposrat}.} (Dense positive rationals.) Let $a,b \in L$ be two unequal positive elements without a rational number between them. Without loss of generality, assume $a<b$. By property~\ref{parsub} there exists a positive element $c \in L$ with $a+c=b$, and by property~\ref{postiny} there exists some natural number $n \in L$ with $nc>1$. It follows that $nb = na+nc>na+1$. Write $p \in L$ for the smallest natural number greater than $na$, which exists by property~\ref{posinf}; it satisfies $na+1>p>na$. Then $nb>na+1>p>na$. Multiplying by the rational $\frac 1 n$ we obtain $b > \frac p n > a$, and we have proved the property.

\emph{\ref{realnum}.} (Real numbers.) For any positive element $a$, define the set $\mathbb{Q} ^+ _{<a}$ to consist of the positive rational numbers strictly less than $a$. From property~\ref{postiny}  there are no infinitesimals and $\mathbb{Q} ^+ _{<a}$ is not empty; also, since it has an upper bound $a$ it has a least upper bound $\bigvee(\mathbb{Q} ^+ _{<a})$. Suppose $\bigvee(\mathbb{Q} ^+ _{<a}) < a$; then by property~\ref{denseposrat} there exists some rational element $r$ satisfying $\bigvee( \mathbb{Q} ^+ _{<a})<r<a$. But this contradicts the definition of $\bigvee(\mathbb{Q} ^+ _{<a})$, and we conclude that $\bigvee(\mathbb{Q} ^+ _{<a}) = a$. We immediately obtain an isomorphism between the nonnegative elements of $L$ and the positive real numbers $\mathbb{R} ^{\geq 0}$, since any positive real number is the supremum of the positive rationals below it.

Suppose that the nonnegative elements do not comprise the entire semiring; then there exists some $b \in L$ with $b<0$. Then $b^2>0$, and identifying $b^2$ with a real number, we can find a positive element $c \in L$ with $c ^2 = b ^2$, and a positive element $\frac 1 c \in L$ which is the reciprocal of $c$. Then defining $x = \frac b c + 1$, we see that
\begin{equation}
\textstyle
x^2 = \big( \frac b c + 1 \big) ^2 = \big( \frac {b^2} {c^2} + 1 + 2 \frac b c \big) = 2+2 \frac b c = 2\big(1+ \frac b c \big) = 2x.
\end{equation}
Suppose that $x \neq 0$; from the multiplicative cancellability property this implies that $x= \frac b c + 1 = 2$, and therefore that $b = c$. But this is not possible, since $b<0$ and $c>0$. We conclude that $x=0$, and therefore that $\frac b c + 1 = 0$ and $\frac b c = -1$. It follows that the semiring is in fact a ring, and that the negative elements are in bijection with the positive elements under multiplication by $-1$. We therefore obtain an isomorphism between the entire ring and the real numbers $\mathbb{R}$ by the method described in the previous paragraph, and it is clear that our semiring is not only a ring, but a field.
\end{proof}

We now combine this lemma with Theorem~\ref{fieldtheorem} to prove out main result, which demonstrates the existence of complex numbers in a category based only its completeness properties. Note that the  statement of this theorem only makes sense in the light of Theorem~\ref{fieldtheorem}, which guarantees that the self-adjoint scalars will admit a total order compatible with the semiring structure.
\begin{theorem}
\label{complextheorem}
In a monoidal \dag\-category  with simple tensor unit, which has all finite \dag\-limits, and for which the self-adjoint scalars are Dedekind-complete, the scalars have an involution-preserving embedding into the complex numbers.
\end{theorem}

\jbeginproof
Writing $S$ for the semiring of scalars, we write $L \subseteq S$ for the subsemiring of self-adjoint scalars. This semiring is commutative by Lemma~\ref{eh}, contains the positive rational numbers by Lemma~\ref{fraclemma}, is additively cancellable by Lemma~\ref{additivelemma}, is multiplicatively cancellable by Lemma~\ref{lemmacancmult}, admits a total ordering by Theorem~\ref{fieldtheorem}, and in fact admits an addition-compatible Dedekind-complete ordering by hypothesis. Lemma~\ref{semiringproperties} therefore applies and $L$ is either $\mathbb{R} ^{\geq 0}$ or $\mathbb{R}$, the latter being the smallest field into which $L$ embeds. It follows that $Q(D(L))=D(L) = \mathbb{R}$, where $Q(-)$ and $D(-)$ construct the smallest field containing a particular ring and and smallest ring containing a particular semiring respectively, in the manner described in section~\ref{mainsection}.

By Theorem~\ref{fieldtheorem} we know that $S$ has an involution-preserving embedding into the involutive field $Q(D(S))$, and it follows immediately that the subsemiring $L$ has an embedding into $F \subseteq Q(D(S))$, the subfield consisting of the self-adjoint elements. In fact, this embedding is surjective, as we now show. Consider some element $r=a-b \in D(S)$ where $a,b \in S$; if $r$ is self-adjoint, then this implies $a + b ^\sdag = b + a ^\sdag$. But since $r = (a+a^\sdag) - (a ^\sdag + b)$ we see that $r$\ can be expressed as the difference of elements of $L$, and so the self-adjoint subring of $D(S)$ is precisely $D(L)$. Now consider an element $s \in F \subseteq Q(D(S))$, so $s = c/d$ as a formal ratio of elements $c,d \in D(S)$. If $s$ is self-adjoint then $c ^\sdag / d ^\sdag = c/d$ as formal ratios, which means that $c ^\sdag d = c d ^\sdag$ in $D(S)$. But then we can write $c/d = c d ^\sdag / d d ^\sdag$, demonstrating that $c/d$ is in fact a ratio of self-adjoint elements of $D(S)$, which are precisely elements of $D(L)$. We therefore see that, as subsets, $Q(D(L)) = F \subseteq Q(D(S))$. In particular, since $Q(D(L)) = \mathbb{R}$ we have $F = \mathbb{R}$, and we will use this identification freely in the rest of the proof.

\newcommand\re{\ensuremath{\mathrm{Re}}}
\newcommand\im{\ensuremath{\mathrm{Im}}}
We will demonstrate an involution-preserving embedding of $Q(D(S))$ into the complex numbers. Since $L$ is either $\mathbb{R} ^{\geq 0}$ or $\mathbb{R}$, then $Q(D(L)) = D(L) = \mathbb{R}$. Suppose that the involution on the scalars is trivial; then $L=S$, and $Q(D(S)) = Q(D(L)) = \mathbb{R} \subset \mathbb{C}$, so the theorem holds. Otherwise, let $x \in Q(D(S))$ be an element of our field such that $x ^\sdag \neq x$; then $y:=x-x^\sdag$ is a nonzero element satisfying $y ^\sdag = -y$, and $y^\sdag y \in F$ is a nonzero real number. Suppose that $y ^\sdag y< 0$; then $-y ^\sdag y$ is a positive real number with a positive root $r \in F$ satisfying $r ^\sdag r + y ^\sdag y = 0$. But by Lemma~\ref{sumnonzerosquares} this cannot be the case, and we conclude that $y ^\sdag y > 0$. Let $s \in F$ be the positive root of $y^\sdag y$ satisfying $s^2 = y ^\sdag y$, and define $j = y/s$. Then $j ^\sdag = y ^\sdag / s ^\sdag = -y/s=-j$ and $j ^2 = y^2/s^2 = -y^\sdag y/s^2 = -1$, and $j$ satisfies the properties that we expect of $i \in \mathbb{C}$.
With this in mind, for all elements $z \in Q(D(S))$ we define  $\re(z)$, $\im(z) \in F$ by
\begin{align}
\re(z) &= \textstyle \frac 1 2 (z + z ^\sdag),
\\
\im(z) &= \textstyle \frac 1 {2j} (z - z ^\sdag).
\end{align}
These are the unique elements of $F$ such that \mbox{$z = \re(z) + j \hspace{0.25pt} \im(z)$}. From this decomposition we obtain an obvious field homomorphism $\sigma:Q(D(S)) \to \mathbb{C}$ given by $\sigma(z) = \re(z) + i \hspace{0.25pt} \im(z)$, where $i \in \mathbb{C}$ is a square root of $-1$, and where we are using the identification of $F$\ with the real numbers. This homomorphism is clearly injective, and it is surjective since any element $k \in \mathbb{C}$ is equal to $\sigma(\Re(k) + j \Im(k))$, so it is a field isomorphism. Since the semiring $S$ has an involution-preserving embedding into $Q(D(S))$, it therefore also has an involution-preserving embedding into $\mathbb{C}$, with involution given by complex conjugation.

Finally we will show that if the involution on the scalars is nontrivial, then the scalar semiring is actually isomorphic to $\mathbb{C}$, with involution given by complex conjugation. We have demonstrated the existence of an involution-preserving embedding of $S$ into $\mathbb{C}$, and in what follows we will use this embedding freely. Since we know that $S$ at least contains $\mathbb{R} ^{\geq 0}$, we only need to show that it also contains $i$, since it will then contain the entire complex plane. Suppose some nonzero element $a \in S$ has $\Re(a)=0$; then $a = i r$ for some $r \in \mathbb{R} \subset \mathbb{C}$. Write $r ^+ \in \mathbb{R} ^{\geq 0}$ for the positive root of $r ^2$; then since $\mathbb{R} ^{\geq 0} \subset S$, we have $1/r ^+ \in \mathbb{R} ^{\geq 0} \subset S$. It follows that $a(1/r ^+) = \pm i \in S$, and so either this quantity or its adjoint is $i \in S$. So, if we can show the existence of a nonzero element of $S$\ with zero real part, our result will follow. We know that there exists some $b \in S$ with $b \neq b ^\sdag$. Suppose $\Re(b) =0$; then we are done. Suppose instead that $\Re(b) < 0$; then defining $c := - \Re(b) \in \mathbb{R} ^{\geq 0} \subset S$ we see that $\Re(b+c) =0$, so we are done. Finally, suppose that $\Re(b) > 0$; then from a simple consideration of the geometry of the complex plane, it is straightforward to see that there exists some natural number $n$ with $b ^n \in S$ such that $\Re(b ^n) < 0$, but we just demonstrated that the existence of such an element implies $i \in S$. We conclude that whatever the value of $\Re(b)$ we have $i \in S$, and so the involutive semiring $S$ can be identified with the field $\mathbb{C}$, with involution given by complex conjugation.
\end{proof}

\jnoindent
In particular, the scalars can be identified with either $\mathbb{R} ^{\geq 0}$ or $\mathbb{R}$ with trivial involution,  or $\mathbb{C}$ with complex conjugation as involution.

\section{Categorical description of inner products}

In this section, we will see how \dag\-limits can be used to define the \dag\-functor on the category \cat{FdHilb} of finite-dimensional Hilbert spaces. Since knowing the \dag\-functor on this category is equivalent to knowing the inner products on all the objects, we also obtain a new way to describe inner products.

We begin with a useful technical lemma. If \cat C and \cat D are \dag\-categories and $F: \cat C \to \cat D$ is a functor, then we say that $F$ \emph{commutes with the \dag\-functors} if $\dag \circ F = F \circ \dag$, where the first $\dag$ is on the category \cat D and the second is on the category \cat C. Also, we recall the definition of a \emph{unitarily essentially surjective functor} as a functor with every object in the codomain unitarily isomorphic to some object in the functor's image, and \textit{unitary \dag\-equivalence} as an equivalence between two \dag\-categories which commutes with the \dag\-functors, and for which the natural isomorphisms are unitary at every stage. 

\begin{lemma}
\label{ude}
Suppose that that there is a functor between two \dag\-categories which is full, faithful, unitarily essentially surjective, and commutes with the \dag\-functors. Then it forms part of a unitary \dag\-equivalence.
\end{lemma}

\jbeginproof
We prove this by extending the conventional argument~\mbox{\cite[Theorem IV.4.1]{ml97-cwm}} that a full, faithful and essentially surjective functor forms part of an equivalence. Suppose that a functor $F: \cat C \to \cat D$ has the properties described in the hypothesis. Then for any object $d \in \cat D$, we can find a object $G_0 (d) \in \cat\ C$ and an unitary morphism $\eta _d: d \to F( G_0 (d))$. We want to promote the function $G_0 : \mathrm{Ob}( \cat D ) \to \mathrm{Ob} ( \cat C )$ into a functor $G:D \to C$, such that $\eta$ becomes a natural transformation. The naturality square for $\eta$ looks like this:
\begin{equation}
\begin{diagram}[midshaft,width=50pt,height=35pt]
d & \rTo^f & d'
\\
\dTo < {\eta_d} && \dTo > {\eta _{d'}}
\\
FG(d) & \rTo ^{FG(f)} & FG(d')
\end{diagram}
\end{equation}
It follows that $FG(f) = \eta_d ^\dag; f; \eta _{d'} ^{\pdag}$, and since $F$ is full and faithful, this uniquely defines $G$. Constructing this equation for the adjoint of $f$ we have $FG(f ^\dag) = \eta _{d'} ^\dag ; f ^\dag ; \eta _d ^\pdag$, and taking the adjoint of this equation gives $F(G(f ^\dag) ^\dag) = \eta _{d} ^\dag; f ; \eta_{d'} ^\pdag$. It follows that $FG(f) = F(G(f ^\dag) ^\dag)$, and since $F$ is full and faithful $G(f ^\dag) = G(f) ^\dag$, so $G$ commutes with the \dag\-functors. To fully demonstrate the unitary \dag\-equivalence we still need to construct a unitary natural transformation $\epsilon : GF \Rightarrow \id _{\cat C}$. We define this by $F(\epsilon _c) = \eta _{Fc} ^\dag$; since $F$ is full and faithful, this definition is valid. It is easy to show that these morphisms are unitary and natural, and in fact, the equivalence is an adjoint equivalence.
\end{proof}

We now prove the main theorem of this section.

\begin{theorem}
\label{innerproducttheorem}
Let $\dag:\cat{FdVect} \to \cat{FdVect}$ be a \dag\-functor on the monoidal category of finite-dimensional complex vector spaces. Then the following properties are equivalent:
\jbeginenumerate
\item equipped with \dag, \cat{FdVect} has all finite \dag\-limits and Dedekind-complete self-adjoint scalars;
\item there is a choice of inner product on each object of \cat{FdVect} such that the \dag\-functor acts by taking adjoints with respect to these inner products;
\item
there is a unitary \dag\-equivalence between \cat{FdVect} with its specified \dag\-functor, and \cat{FdHilb} with its canonical \dag\-functor.
\end{enumerate}
\end{theorem}

\jbeginproof

We begin with the implication $1 \Rightarrow 2$. The complex numbers are present in \cat{Vect} as endomorphisms of the one-dimensional vector space, and the \dag\-functor gives it an involution; we denote this involutive field by $(\mathbb{C}, \dag)$. This could be different to $(\mathbb{C}, {*})$, the complex numbers equipped with complex conjugation as involution. However, by Theorem~\ref{complextheorem}, there must be an involution-preserving field isomorphism  $\chi:(\mathbb{C},\dag) \to (\mathbb{C}, *)$. Since $\chi$ preserves the involution we have $\chi \dag = {*}\chi $, and since $\chi$ is invertible, we see that the involution induced by the \dag\-functor is \textit{conjugate} to complex conjugation.

For every object $A$ in \cat{Vect} we define a putative inner product for all $\phi,\psi:\mathbb{C} \to A$ as $\langle \phi, \psi \rangle := \chi(\psi ; \phi ^\dag): \mathbb{C} \to \mathbb{C}$. We must show that this satisfies the axioms of an inner product. We first establish that ${*}(\langle \phi,\psi \rangle) = \langle \psi, \phi \rangle$, by observing that $*(\langle \phi,\psi \rangle) = {*} \chi(\psi ; \phi ^\dag) = \chi {\dag} (\psi ; \phi ^\dag) = \chi(\phi; \psi ^\dag) = \langle \psi,\phi \rangle$. Now, suppose that some vector $\phi: \mathbb{C} \to A$ has negative norm under this inner product; without loss of generality we assume that it is normalized, so that $\phi$ satisfies $\langle \phi, \phi \rangle = -1$. Now consider the column vector $( \!\begin{smallmatrix} 1 \\[-0.5pt] \phi \end{smallmatrix}\! ): \mathbb{C} \to \mathbb{C} \oplus A$; this will have a norm of zero, which is ruled out by the \dag\-equalizer property as established by Lemma~\ref{nondegen}. We conclude that $\langle \phi, \phi \rangle > 0$ for all nonzero $\phi$. Linearity of the inner product follows straightforwardly from the properties of \dag\-biproducts. Altogether, the construction $\langle \phi, \psi \rangle := \chi( \psi ; \phi ^\sdag)$ is linear in the second argument, conjugate-symmetric and positive-definite, and hence is a genuine inner product. It is then trivial that for all $f:A \to B$, $\phi_A:\mathbb{C} \to A$ and $\phi_B: \mathbb{C} \to B$, we have $\langle\phi _B, (\phi _A ; f) \rangle = \langle (\phi _B; f ^\dag ), \phi _A \rangle$, and so the \dag\-functor takes linear maps to their adjoints and we have proved the implication.

For the implication $2 \Rightarrow 3$, the choice of functor is obvious: every object of \cat{FdVect} has an assigned inner product, and since a finite-dimensional complex vector space with inner product is necessarily a Hilbert space, we have a functor into \cat{FdHilb}. This functor is full, faithful and essentially surjective, since every Hilbert space is defined up to isomorphism by its cardinality and there will be Hilbert spaces of every finite cardinality in the image of the functor. Finally, it is clear that the inclusion is compatible with the action of the \dag\-functor, and since two isomorphic Hilbert spaces always have a unitary isomorphism between them, we have a unitary \dag\-equivalence by Lemma~\ref{ude}.

Finally, we consider the implication $3 \Rightarrow 1$. Let $F: \cat{FdHilb} \to \cat{FdVect}$ be  a functor forming part of the unitary \dag\-equivalence; then it gives rise to an involution-preserving field homomorphism $F: (\mathbb{C}, {*}) \to (\mathbb{C}, {\dag})$, where ${*}$ is the complex conjugation operation and ${\dag}$ represents the action of the \dag\-functor on the scalars of \cat{FdVect}. Since $F$\ also gives rise to a field isomorphism between the self-adjoint elements of both fields, and since the self-adjoint elements of $(\mathbb{C}, {*})$ are Dedekind-complete under the unique order on the real numbers, it follows that the self-adjoint elements of $(\mathbb{C}, \dag)$ also admit a unique order, which is Dedekind-complete. The implication is completed with the straightforward fact that, just as limits are preserved by equivalences, \dag\-limits are preserved by unitary \dag\-equivalences.
\end{proof}

\jnoindent
A similar theorem would hold for the category of all complex vector spaces, but we would then be dealing with inner-product spaces rather than Hilbert spaces.

\section{Technical discussion}

Our results give an abstract characterization of the properties endowed by the complex numbers on a physical theory. More importantly, this abstract characterization --- formalized by Theorem~\ref{complextheorem} --- admits a relatively clear physical interpretation. The most important structure is the requirement of having all finite  \dag\-limits, a type of completeness property which can be interpreted as the ability to take the direct sum of separate physical systems, modulo the action of processes, in a way which preserves norms. Another crucial structure is Dedekind completeness, which is the requirement that, in the totally-ordered set of self-adjoint scalars, every bounded set has a least upper bound and a greatest lower bound, and that these bounds get along with addition of scalars. In conventional quantum physics these self-adjoint scalars represent the results of measurements, and Dedekind completeness is a property that we observe experimentally. The final property is that the theory has a simple tensor unit; physically, this means that there exists a `trivial system' that behaves in sensible way, such that the only smaller system is the empty system.

The most arguable physical property is perhaps that of Dedekind completeness. Even without this Theorem~\ref{fieldtheorem} still applies, telling us that the theory is built on an involutive field of characteristic~0, with an orderable fixed field which is not the real numbers.

It is interesting to consider  the role played by the \dag\-functor in these results, which represents our ability to turn any process $f:A \to B$ into a process $f^\sdag :B \to A$. Two important lemmas, the cancellable addition lemma~\ref{additivelemma} and the exchange lemma~\ref{exchangelemma}, seem to rely crucially on the \dag\-functor. It seems that the power of the \dag\-functor lies in its ability to add an extra degree of symmetry to a system of equations. For example, in the proof of Lemma~\ref{additivelemma}, the role of the \dag\-functor is to prove $e_1 ^\sdag; e ^\pdag _2 = 0$ from the known equation $e_2 ^\sdag; e_1 ^\pdag = 0$. The underlying \dag\-equalizer diagram does not have a symmetry exchanging $e_1$ and $e_2$, but the existence of the \dag\-functor forces the existence of such a symmetry, proving the theorem. This contrasts with the proof of Lemma~\ref{ABablemma}, for which the diagram does have a symmetry exchanging $e_1$ and $e_2$, and the \dag\-functor is not directly required for the proof.

It is possible to consider variants of \dag\-biproducts and \dag\-equalizers that do not rely on the \dag\-functor --- such as biproducts, and equalizers that have retractions --- but it does not seem that these would be powerful enough to prove analogous results. We hope that these results, and others that rely crucially on properties of the \dag\-functor (such as~\cite{cpv08-dfb, dr89-ndt, v08-cfqa}), will stimulate interest in the \dag\-functor as a fundamental construction, both in category theory and in the foundations of quantum theory.

\bibliography{../../../jov}

\end{document}